\documentclass[11pt,reqno]{amsart}
\usepackage{amssymb}
\usepackage{mathptmx}
\usepackage{mathtools}
\usepackage{mathrsfs}
\usepackage[all]{xy}
\usepackage{stmaryrd}
\usepackage{fancyhdr}
\usepackage{hyperref}
\usepackage{mathrsfs}
\usepackage{tikz-cd}
\usepackage{cite}
\usepackage{amsmath,amsfonts}
\usepackage{graphicx}
\usepackage{wrapfig}
\usepackage{array}
\usepackage{amsthm}
\usepackage[left=1.2in, right=1.2in, top=1in, bottom=1in]{geometry}
\usepackage{etoolbox}
\usepackage{hyperref}
\patchcmd{\section}{\scshape}{\bfseries}{}{}
\makeatletter
\renewcommand{\@secnumfont}{\bfseries}
\makeatother
\xyoption{all}
\theoremstyle{plain}

\newcommand*{\Scale}[2][4]{\scalebox{#1}{$#2$}}

\newcommand{\hooklongrightarrow}{\lhook\joinrel\longrightarrow}

\DeclareMathOperator{\Hom}{Hom}
\DeclareMathOperator{\Trop}{Trop}

\DeclareMathOperator{\Spec}{Spec}
\DeclareMathOperator{\Sper}{Spec_r}
\DeclareMathOperator{\Frac}{Frac}

\DeclareMathOperator{\Aut}{Aut}

\DeclareMathOperator{\red}{\mathbf{red}}
\input xy
\xyoption{all}

\theoremstyle{definition}
\newtheorem{mydef}{\textbf{Definition}}[section]
\newtheorem{myeg}[mydef]{\textbf{Example}}

\newtheorem{mythm}[mydef]{\textbf{Theorem}}

\newtheorem{question}[mydef]{\textbf{Question}}
\newtheorem{rmk}[mydef]{\textbf{Remark}}

\theoremstyle{plain}
\newtheorem*{nothm}{\textbf{Theorem}}

\newtheorem{mytheorem}[mydef]{\textbf{Theorem}}
\newtheorem{lem}[mydef]{\textbf{Lemma}}
\newtheorem{pro}[mydef]{\textbf{Proposition}}

\newtheorem{cor}[mydef]{\textbf{Corollary}}

\begin{document}

\title{Geometry of hyperfields}
\author{Jaiung Jun}
\address{Department of Mathematics, State University of New York at New Paltz, New Paltz, NY 12561, USA}
\curraddr{}
\email{junj@newpaltz.edu}

%\date{July 24, 2017}

\dedicatory{}

\maketitle
%\tableofcontents

\begin{abstract}
Given a scheme $X$ over $\mathbb{Z}$ and a hyperfield $H$ which is equipped with a topology which satisfies certain conditions, we endow the set $X(H)$ of $H$-rational points with a natural topology. We then prove that; (1) when $H$ is the \emph{Krasner hyperfield}, $X(H)$ is homeomorphic to the underlying space of $X$, (2) when $H$ is the \emph{tropical hyperfield} and $X$ is of finite type over a complete non-Archimedean valued field $k$, $X(H)$ is homeomorphic to the underlying space of the Berkovich analytificaiton $X^{\textrm{an}}$ of $X$, and (3) when $H$ is the \emph{hyperfield of signs}, $X(H)$ is homeomorphic to the underlying space of the real scheme $X_r$ associated with $X$. 
\end{abstract}

\tableofcontents

\section{Introduction}

A \emph{hypergroup} assumes similar axioms as an abelian group except that one allows addition to be \emph{`multi-valued'} (hyperaddition). A \emph{hyperring} $R$ is a nonempty set with two binary operations (multiplication $\cdot$, hyperaddition $+$) such that $(R,\cdot)$ is a commutative monoid, $(R,+)$ is a hypergroup, and hyperaddition is distributive over multiplication. When all nonzero elements are multiplicatively invertible, a hyperring is called a \emph{hyperfield}.

An early incarnation of \emph{hyperstructures} goes back to F.~Marty \cite{marty1935role} who introduced the notion of \emph{hypergroups}. After about 20 years, M.~Krasner adapted Marty's idea to generalize commutative rings to \emph{hyperrings} in \cite{krasner1956approximation}. Krasner's goal was to approximate (in a precise sense) Galois theory of local fields of positive characteristic by means of Galois theory of local fields of characteristic zero. Since its first appearance, hyperrings have not been sufficiently studied from algebro-geometric nor combinatorial perspective. However, lately there has been considerable attention drawn to the theory of hyperrings under various motivations mainly thanks to the following pioneering work: (1) A.~Connes and C.~Consani provide several evidences which show that hyperrings are algebraic structures which naturally appear in relation to algebraic geometry and number theory \cite{con5,con4,con3,con6}, (2) M.~Marshall, P.~G\l{}adki, and K.~Worytkiewicz show that hyperstructures simplify (or generalize) certain aspects of quadratic form theory and real algebraic geometry, \cite{mars1,mar2,gladki2017witt}, (3) O.~Viro implements hyperfields to tropical geometry to provide an algebraic foundation \cite{viro}, and (4) M.~Baker and N.~Bowler unify various generalizations of matroids by means of one unified framework, namely matroids over partial hyperstructures \cite{baker2016matroids}. Baker and Bowler's work additionally supports Viro's philosophy in perspective of valuated matroids. For more details about this viewpoint in connection with O.~Lorscheid's blueprints, see \cite{lorscheid2019tropical}.\\

In this paper, for a scheme $X$, we investigate the set $X(H)$ of rational points over a hyperfield $H$ by appropriately generalizing the notion of locally ringed spaces to locally hyperringed spaces, and $X(H)$ is defined to be a set of morphisms from ``$\Spec H$'' to $X$ in this category. In general, $X(H)$ is merely a set; however, when $H$ is equipped with a topology which satisfies certain conditions, we impose the \emph{fine topology} on $X(H)$, introduced by O.~Lorscheid and C.~Salgado in \cite{lorscheid2016remark} (and also in \cite{baker2018moduli}). %Note that the authors of \cite{lorscheid2016remark} called it \emph{fine Zariski topology}.\\

A hyperfield $H$ is said to be a \emph{topological hyperfield} if $H$ is equipped with topology, satisfying the following conditions:
\begin{itemize}
\item 
The multiplication map $H \times H \to H$, where $H \times H$ is equipped with product topology, is continuous. 
\item 	
$H^\times=H-\{0\}$ is open, and the inversion map $i:H^\times \to H^\times$ is continuous. 
\end{itemize}

Note that in the definition of a topological hyperfield, we do not assume the compatibility of hyperaddition with a topology. This could look odd, however, the above conditions are enough to fulfill our purpose in this paper. We further note that in their work \cite{anderson2019hyperfield} on hyperfield Grassmannians, L.~Anderson and J.~Davis also assume the same conditions (as above) for topological hyperfields (see, \cite[Remark 2.13]{anderson2019hyperfield}). In what follows, we will always assume that any hyperfield, which is equipped with a topology, is a topological hyperfield.\\

%We do not assume any compatibility of the algebraic structures of $H$ with $\mathcal{T}$. 

Now, let us briefly recall the definition of the fine topology. Let $X=\Spec A$ be an affine scheme and $H$ a topological hyperfield. Then, one imposes a canonical topology (\emph{affine topology}, see \cite{lorscheid2016remark}) on $\Hom(A,H)$, the set of hyperring homomorphisms from $A$ to $H$; we first consider the following set-inclusion
\[
\Hom(A,H) \hooklongrightarrow \prod_{a \in A}H^{(a)}
\]
and give $\Hom(A,H)$ subspace topology, where $\prod_{a \in A}H^{(a)}$ is equipped with product topology induced by a topology $\mathcal{T}$ of $H$. For the general case, when $X$ is a scheme and $H$ is a topological hyperfield, the \emph{fine topology} on the set $X(H)$ of $H$-rational points of $X$ is the finest topology such that any morphism $f:Y \to X$ from an affine scheme $Y$ to $X$ induces a continuous map $f(H):Y(H) \to X(H)$, where $Y(H)$ is equipped with the affine topology. One can easily check that the fine topology agrees with the affine topology when $X$ is affine (see, \S \ref{finetopology}).

With this setting, the main question which we want to address is the following:

\begin{question}\label{question}
Which topological space arises as the set $X(H)$ of rational points of an algebraic variety $X$, where $H$ is a topological hyperfield? Do we have some interesting known-examples of $X(H)$?
\end{question}

In this paper, we will prove that several familiar topological spaces arise in this way. To this end, topological hyperfields of particular interest are the following (see, \S \ref{definitions} for details):
\begin{itemize}
\item (Krasner hyperfield) Let $\mathbb{K}:=\{0,1\}$ be a commutative monoid with the multiplication $1\cdot 0=0$ and $1\cdot 1=1$. The hyperaddition is given by $0+1=1+0=1$, $0+0=0$, and $1+1=\{0,1\}$. Then $\mathbb{K}$ is a hyperfield called the \emph{Krasner hyperfield}. We impose topology on $\mathbb{K}$ in such a way that the set of open subsets is $\{\emptyset, \{1\}, \mathbb{K}\}$. 
\item(Tropical hyperfield) Let $\mathbb{T}:=\mathbb{R}\cup \{-\infty\}$, where $\mathbb{R}$ is the set of real numbers. The multiplication $\odot $ is the usual addition of $\mathbb{R}$ such that $a\odot (-\infty)=(-\infty)$ for all $a \in \mathbb{T}$. The hyperaddition $\oplus$ is to take the maximum when two elements are different, i.e., $a \oplus b =\max\{a,b\}$ if $a\neq b$. When $a=b$, $a\oplus b=\{c \in \mathbb{T} \mid c \leq a\}$, where $\leq$ is the usual order of $\mathbb{R}$ with $-\infty$ the smallest element. Then $\mathbb{T}$ is a hyperfield called the \emph{tropical hyperfield}. We simply impose Euclidean topology on $\mathbb{T}$.  
\item (Hyperfield of signs) Let $\mathbb{S}:=\{-1,0,1\}$ be a commutative monoid with the multiplication $1\cdot 1 =1$, $(-1)\cdot (-1)=1$, $(-1)\cdot 1=(-1)$, and $1\cdot 0 =(-1)\cdot 0 =0\cdot 0 =0$. The hyperaddition follows the rule of signs, i.e., $1+1=1$, $(-1)+(-1)=(-1)$, $1+0=1$, $(-1)+0=(-1)$, and $1+(-1)=\{-1,0,1\}$. Then $\mathbb{S}$ is a hyperfield called the \emph{hyperfield of signs}. We impose topology on $\mathbb{S}$ in such a way that the set of open subsets is $\{\emptyset, \{1\},\{-1\}, \{-1,1\}, \mathbb{S}\}$. 
\end{itemize}
\vspace{0.2cm}

%Roughly speaking, Baker and Bowler prove the following among others. 

%\begin{nothm}\cite{baker2016matroids}
%There is a notion of a matroid $M$ over a hyperfield $H$ such that:
%\begin{enumerate}
%\item 
%When $H=K$, a field, then $M$ is a linear space. 
%\item
%When $H=\mathbb{K}$, the Krasner hyperfield, then $M$ is a matroid. 
%\item
%When $H=\mathbb{S}$, the hyperfield of signs, then $M$ is an oriented matroid.
%\item
%When $H=\mathbb{T}$, the tropical hyperfield, then $M$ is a valuated matroid. 
%\item
%When $H=\mathbb{P}$, the phase hyperfield, then $M$ is a phased matroid in the sense of L.~Anderson and E.~Delucchi in \cite{anderson2012foundations}. 
%\end{enumerate}
%\end{nothm} 

\begin{rmk}
\begin{enumerate}
%\item 
%In fact, Baker and Bowler have two notions (strong and weak) of matroids over hyperfields. But, when a hyperfield $H$ satisfies a certain property (perfect hyperfield), these two notions agree. In particular, if a hyperfield $H$ satisfies a doubly-distributive property (see, \cite[\S 5]{baker2016matroids}), then $H$ is perfect. 
%\item
%Since a tropical linear space is the same thing as a valuated matroid (see, \cite[\S 4]{bernd}), the above theorem hints at that hyperstructures may provide an algebraic foundation for tropical geometry.
\item 
The inspiration of the current paper stems from the recent paper \cite{baker2016matroids} of Baker and Bowler, where they develop a very elegant framework which unifies the notions of many enrichments of matroids (matroids, oriented matroids, valuated matroids, and phase matroids) as well as linear spaces at the same time. A key idea is to implement hyperfields in matroid theory to define matroids with coefficients in hyperfields.  
\item 
In \cite{anderson2019vectors}, Anderson provides anther equivalent definition for (strong) matroids over hyperfields in terms of vectors. Furthermore, Anderson and Davis investigate realization spaces for matroids over hyperfields in 
\cite{anderson2019hyperfield}. Also, recently, Baker and Lorscheid explore the moduli space of matroids in \cite{baker2018moduli} as representable functors on certain categories.  
\item 
After I put the current paper on arXiv, in \cite{jell2018real}, P.~Jell, C.~Scheiderer, and J.~Yu recently introduced the notion of real tropicalization and analytification, adding another example of topological spaces which are homeomorphic to the set of $H$-rational points for some hyperfield $H$. 
\end{enumerate}
\end{rmk}

From our point of view, Baker and Bowler's framework could be seen as an investigation of a Grassmannian over various hyperfields (see, the cryptomorphic axiomatization of matroids over hyperfields via Grassmann-Pl\"{u}cker functions in \cite{baker2016matroids}). Hence, one is induced to wonder what we can say when we replace a Grassmannian with other algebraic varieties, or schemes in general (Question \ref{question}). 

\begin{rmk}
One partial answer for Question \ref{question} is initially given by Connes and Consani. In \cite{con3}, Connes and Consani show that there exists a set-bijection between any scheme $X$ over $\mathbb{Z}$ and the set $X(\mathbb{K})$ of $\mathbb{K}$-rational points of $X$, where $\mathbb{K}$ is the Krasner hyperfield. In this paper, we refine this theorem into a homeomorphism.
\end{rmk}

In this paper, we consider three particular topological spaces which appear in algebraic geometry, namely schemes, Berkovich analytification of schemes , and real schemes.
Very roughly, the underlying set of the Berkovich analytification $X^{\textrm{an}}$ of an algebraic variety $X$ over a field $k$ with a complete non-Archimedean valuation $\nu$ consists of pairs of point $x \in X$ and a valuation $\tilde{\nu}$ on the residue field $k(x)$ at $x$ which extends $\nu$ (see \S \ref{berkovich} for some details). 

The real spectrum $\Sper A$ of a commutative ring $A$ is an enrichment of the prime spectrum $\Spec A$ consisting of pairs $(\mathfrak{p},P)$ of a prime ideal $\mathfrak{p}$ of $A$ and an ordering $P$ of the residue field $k(\mathfrak{p})$. Once we globalize this construction, we obtain the real scheme $X_r$ associated to a scheme $X$ (see \S\ref{spaceoforderings} for some details). 

Our main result is that the aforementioned spaces arise as sets of rational points over certain hyperfields in a functorial way. To be precise, we prove the following theorem. 

\begin{nothm}
Let $X$ be a scheme over $\mathbb{Z}$, $k$ be a field with a complete non-Archimedean valuation, $\mathfrak{Schm}$ be the category of schemes over $\mathbb{Z}$, and $\mathfrak{Top}$ be the category of topological spaces. 
\begin{enumerate}
\item 
The set $X(\mathbb{K})$ of $\mathbb{K}$-rational points of $X$ (equipped with the fine topology) is homeomorphic to $|X|$ (the underlying topological space of $X$). In particular, the functor $\mathcal{F}$ from $\mathfrak{Schm}$ to $\mathfrak{Top}$, sending any scheme $X$ to $|X|$, is isomorphic to the functor $\Hom(\Spec \mathbb{K},-)$.  
\item
Let $X$ be a scheme of finite type over $k$. Then the Berkovich analytification $X^{\textrm{an}}$ of $X$ is homeomorphic to $X(\mathbb{T})$ (equipped with the fine topology). In particular, the functor $\mathcal{A}$ from the category $\mathfrak{Schm}_{k,fin}$ of schemes of finite type over $k$ to $\mathfrak{Top}$, sending any scheme $X$ to the underlying topological space $|X^{\textrm{an}}|$ of the analytification $X^{\textrm{an}}$, is isomorphic to the functor $\Hom(\Spec \mathbb{T},-)$.  
\item
The set $X(\mathbb{S})$ of $\mathbb{S}$-rational points of $X$ (equipped with the fine topology) is homeomorphic to the underlying topological space of the real scheme $X_r$ associated to $X$. In particular, the functor $\mathcal{R}$ from $\mathfrak{Schm}$ to $\mathfrak{Top}$, sending any scheme $X$ to the underlying topological space $|X_r|$ of the associated real scheme $X_r$, is isomorphic to the functor $\Hom(\Spec \mathbb{S},-)$.  
\end{enumerate}
\end{nothm}

%\begin{rmk}
%Strictly speaking, since $\Spec H$ is not an object in the category of schemes, the aforementioned functors may not be said ``representable'', however, in this paper, we stick with this language.
%\end{rmk}

When $X$ is a group scheme over a field $k$, the underlying topological space $|X|$ itself is not a group; we only have a group structure for the set $X(K)$ of $K$-rational points for each field extension $K$ of $k$. Therefore, the identification $|X|=X(\mathbb{K})$ naturally leads one to the following question:

\begin{question}
Let $X$ be a group scheme over a field $k$. Is the underlying space $|X|$ itself a hypergroup by viewing it as the set of $\mathbb{K}$-rational points, where $\mathbb{K}$ is the Krasner hyperfield? More generally, if $H$ is a hyperfield, is $X(H)$ a hypergroup?
\end{question}

In \cite{con4}, Connes and Consani prove that the answer is affirmative when $X$ is an affine line or an algebraic torus and $H=\mathbb{K}$ by explicitly describing the hypergroup structures. In \cite{jun2016hyperstructures}, the author also proves that $X(\mathbb{K})$ is ``almost'' a hypergroup (in a precise sense) when $X$ is of finite type. In $\S \ref{hyperstructures of analytic groups}$, we study the case when $H=\mathbb{T}$ and $X$ is of finite type, following the idea of Berkovich in \cite[\S 5]{berkovich2012spectral}. 

Throughout the paper, we assume that all algebraic structures (rings, $k$-algebras, hyperrings, etc) are commutative unless otherwise stated. \\

\textbf{Acknowledgment}
\vspace{0.1cm}

The author would like to thank Vladimir Berkovich and Youngsu Kim for answering various questions concerning $\S \ref{hyperstructures of analytic groups}$. He also thanks Oliver Lorscheid for pointing out minor mistakes which had escaped from the author, and Farbod Shokrieh for answering various questions on Berkovich analytifications. The author is grateful to Jeffrey Giansiracusa for providing helpful feedback for the first draft of the paper. Finally, the author thanks an anonymous referee for many helpful suggestions.

\section{Review: Basic definitions and examples}\label{definitions}

In this section, we recall basic definitions and examples for hyperrings which will be used in the sequel. For more details, we refer the readers to \cite{con3}, \cite{jun2015algebraic}, or \cite{viro}. 

\begin{mydef}
Let $H$ be a nonempty set and $P^*(H)$ be the set of nonempty subsets of $H$.
\begin{itemize}
\item 
By a \emph{hyperoperation} on $H$, we mean a function $*:H\times H \to P^*(H)$. For the notational convenience, we let $a*b:=*(a,b)$ for all $a,b \in H$. 
\item
For $x,y,z \in H$, we define the following two subsets of $H$:
\[
(x*y)*z:=\bigcup_{w\in x*y}w*z, \quad \textrm{and} \quad x*(y*z):=\bigcup_{w\in y*z} x*w. 
\]
When $(x*y)*z=x*(y*z)$ for all $x,y,z \in H$, we say that a hyperoperation $*$ is \emph{associative}. 
\item
When $x*y=y*x$ for all $x,y \in H$, we say that a hyperoperation $*$ is \emph{commutative}. 
\item
When the set $x*y$ consists of a single element $z$, we write $x*y=z$ instead of $x*y=\{z\}$. In general, we will identify an element $x \in H$ with the subset $\{x\}$ of $H$. 
\end{itemize}
\end{mydef}

\begin{rmk}
Let $H$ be a nonempty set with a hyperoperation $*$ and $A$, $B$ be nonempty subsets of $H$. Then we will use the following notation:
\[
A*B:=\bigcup_{a \in A, b\in B} a*b. 
\]
\end{rmk}

\begin{mydef}
A \emph{hypergroup} is a nonempty set $H$ equipped with an associative hyperoperation $*$ such that
\begin{enumerate}
\item (Unique identity)
$\exists ! e\in H$ such that $e*x=x*e$. 
\item (Unique inverse)
For each $x \in H$, $\exists ! y \in H$ such that $e \in (x*y) \cap (y*x)$. We denote $y$ by $x^{-1}$. 
\item (Reversibility)
For each $x,y,z \in H$, if $x \in y*z$ then $y \in x*z^{-1}$ and $z \in y^{-1}*x$. 
\item 
When a hyperoperation is commutative, we call $(H,*)$ a \emph{canonical hypergroup}. In this case, we will use additive notations such as $+, \oplus, \boxplus$. 
\end{enumerate}
\end{mydef}

\begin{rmk}
In \cite{jun2016hyperstructures}, we did not include the reversibility $(3)$ as a part of the definition for hypergroups. 
\end{rmk}

\begin{mydef}
A \emph{hyperring} is a nonempty set $R$ with two binary operations $+$ and $\cdot$ such that $(R,+)$ is a canonical hypergroup and $(R,\cdot)$ is a commutative monoid satisfying the following conditions:
\begin{enumerate}
\item (Distributivity)
$x\cdot(y+z)=x\cdot y +x \cdot z$ for all $x,y,z \in R$. 
\item (Absorbing element)
If $0$ is the identity element with respect to a hyperoperation. Then $x\cdot 0 =0$ for all $x \in R$. We further assume that $0 \neq 1$. 
\end{enumerate}
When each $x \neq 0 \in R$ has a multiplicative inverse, we call $(R,+,\cdot)$ a \emph{hyperfield}. 
\end{mydef}

\begin{myeg}\label{mainexample}
We introduce some examples of hyperfields which yield interesting results in matroid theory. 
\begin{itemize}
\item 
Let $\mathbb{K}:=\{0,1\}$. The multiplicative structure of $\mathbb{K}$ is same as that of $\mathbb{F}_2$, the field with two element. The commutative hyperaddition is defined as follows:
\[
0+0=0, \quad 0+1=1, \quad 1+1=\{0,1\}. 
\]
$\mathbb{K}$ is called the \emph{Krasner hyperfield}. 
\item
Let $\mathbb{S}:=\{-1,0,1\}$. We impose the multiplication following the rule of signs:
\[
1 \cdot 1=1, \quad 1\cdot (-1) =(-1), \quad 1\cdot 0=0\cdot (-1)=0\cdot 0=0. 
\]
The hyperaddition also follows the rule of signs:
\[
1+1=1+0=1, \quad (-1)+(-1)=(-1)+0=(-1), \quad 1+(-1)=\mathbb{S}, \quad 0+0=0. 
\]
$\mathbb{S}$ is called the \emph{hyperfield of signs}. 
\item
Let $\mathbb{T}:=\mathbb{R}\cup \{-\infty\}$. The multiplication $\odot$ of $\mathbb{T}$ is same as the usual addition of real numbers and $-\infty \odot a =-\infty$ for all $a \in \mathbb{T}$. The hyperaddition is given as follows:
\[
x\oplus y =\left\{ \begin{array}{ll}
\max\{x,y\} & \textrm{if $x\neq y$}\\
\left[-\infty,x\right]& \textrm{if $x=y$},
\end{array} \right.
\]
where $\left[-\infty,x\right]:=\{t\in \mathbb{T} \mid t\leq x\}$. $\mathbb{T}$ is called the \emph{tropical hyperfield}. 
\item 
Let $(\Gamma,+)$ be a totally ordered abelian group and $\Gamma_{hyp}:=\Gamma \cup \{-\infty\}$. One can impose two binary operations $\odot$ and $\oplus$ on $\Gamma_{hyp}$ as follows:
\[
a \odot b:=a+b \textrm{ and } a\odot (-\infty)=-\infty, \textrm{for all }a,b \in \Gamma,
\]
\[
a\oplus b =\left\{ \begin{array}{ll}
\max\{a,b\} & \textrm{if $a\neq b$}\\
\left[-\infty,a\right]& \textrm{if $a=b$},
\end{array} \right. \textrm{ where } \left[-\infty,a\right]:=\{t\in \Gamma_{hyp} \mid t\leq a\}.
\]
Then $(\Gamma_{hyp},\odot,\oplus)$ is a hyperfield. In particular, if $\Gamma=\mathbb{R}$, then $\Gamma_{hyp}=\mathbb{T}$.

\item
Let $\mathbb{P}:=S^1 \cup \{0\}$, where $S^1$ is the unit circle in the complex plane. The multiplication of $\mathbb{P}$ is induced from the multiplication of complex numbers, and we define the hyperaddition as follows:
\[
x\oplus y =\left\{ \begin{array}{ll}
\{-x,0,x\} & \textrm{if $x= -y$}\\
\textrm{ the shorter open arc connecting $x$ and $y$}& \textrm{if $x \neq -y$},
\end{array} \right.
\]
$\mathbb{P}$ is called the \emph{phase hyperfield}. 
\end{itemize}
\end{myeg}

There is a recipe to produce a hyperring from a commutative ring $A$; let $A$ be a commutative ring and $G$ be a subgroup of the multiplicative group $A^\times$ of units in $A$. Then $G$ acts (by multiplication) on $A$. Let $A/G$ be the set of equivalence classes under the action of $G$ and $[a]$ be the equivalence class of $a \in A$. One defines the following binary operations:
\[
[a]\cdot[b]=[ab], \quad [a]+[b]=\{[c] \mid c=g_1a+g_2b \textrm{  for some } g_1,g_2 \in G\}. 
\]
Then $(A/G,+,\cdot)$ is a hyperring (a \emph{quotient hyperring}).

\begin{mydef} (\emph{Homomorphisms of hyperrings})
\begin{enumerate}
\item
Let $H_1$ and $H_2$ be hypergroups with the identities $e_1$ and $e_2$ respectively. A \emph{homomorphism} of hypergroups is a function $f:H_1\to H_2$ such that $f(e_1)=e_2$ and $f(a*b)\subseteq f(a)*f(b)$ for all $a,b \in H_1$. When a homomorphism $f$ satisfies the stronger condition $f(a*b)=f(a)*f(b)$ for all $a,b \in H_1$, we call $f$ \emph{strict}.  
\item
Let $R_1$ and $R_2$ be hyperrings. A \emph{homomorphism} of hyperrings is a function $f:R_1 \to R_2$ such that $f:(R_1,+)\to (R_2,+)$ is a homomorphism of hypergroups and $f:(R_1,\cdot)\to (R_2,\cdot)$ is a homomorphism of monoids. When $f:(R_1,+)\to (R_2,+)$ is a strict homomorphism of hypergroups, we call $f$ \emph{strict}. 
\end{enumerate}
\end{mydef}

\begin{myeg}
Let $R$ be a hyperfield. Then there exists a \emph{unique homomorphism} from $R$ to the Krasner hyperfield; $\pi:R \to \mathbb{K}$ sending $a\neq 0$ to $1$ and $0$ to $0$. In other words, $\mathbb{K}$ is the final object in the category of hyperfields.
\end{myeg}

\begin{rmk}
There are other algebraic structures which generalize commutative rings. In fact, A.~Dress and W.~Wenzel introduce in \cite{dress1991grassmann} the notion of \emph{fuzzy rings} to unify various generalizations of matroids (as in \cite{baker2016matroids}) and also utilize fuzzy rings to recast basic definitions of tropical varieties in \cite{dress2011algebraic}. In \cite{giansiracusa2016relation}, together with J.~Giansiracusa and O.~Lorscheid, we clarify the relation between hyperrings and fuzzy rings to link Baker-Bowler theory and Dress-Wenzel theory. Also, in \cite{rowen2016algebras}, L.~Rowen introduces the notion of \emph{systems} and \emph{triples} to generalize both fuzzy rings and hyperrings along with other algebraic structures. 
\end{rmk}

\section{Algebraic geometry over hyperrings}
\subsection{Locally hyperringed spaces and hyperring schemes}
We first review some basic definitions and properties for \emph{locally hyperringed spaces} and \emph{hyperring schemes} studied in \cite{jun2015algebraic}. We also slightly generalize some of important results in \cite{jun2015algebraic} which serve as main technical tools in this paper.

Baker and Bowler made a remark in \cite{baker2016matroids} that a (semi)valuation (resp. an ordering) on a commutative ring $A$ can be interpreted as a homomorphism from $A$ to the tropical hyperfield $\mathbb{T}$ (resp. the hyperfield $\mathbb{S}$ of signs). The main reason to introduce locally hyperringed spaces and hyperring schemes is to deal with the general case when a scheme is not affine. In other words, a homomorphism from a commutative ring $A$ to a hyperfield $H$ should be replaced by a morphism $\Spec H \to \Spec A$ of locally hyperringed spaces to extend Baker and Bowler's observation to the non-affine case. All necessary details, which we omit in this section, can be found in \cite{jun2015algebraic}. 

\begin{mydef}\label{primeideal}
Let $R$ be a hyperring.
\begin{enumerate}
\item 
An \emph{ideal} $I$ of $R$ is a sub-hypergroup (a subset which is a hypergroup itself with the induced hyperoperation and the identity) of $R$ such that $RI\subseteq I$. 
\item
One defines a \emph{prime ideal} to be the kernel of a homomorphism $\varphi:R\to\mathbb{K}$, where $\mathbb{K}$ is the Krasner hyperfield.
\item
A \emph{maximal ideal} is a proper ideal $\mathfrak{m}$ of $R$, that is $\mathfrak{m}\neq R$, which is not properly contained in any other proper ideal of $R$. 
\end{enumerate}
 
\end{mydef}

\begin{rmk}
One may define a prime ideal of $R$ as in the classical definition; an ideal $\mathfrak{p}$ of $R$ such that $(R\backslash \mathfrak{p},\cdot)$ is a multiplicative monoid. But, one can easily show that this definition is equivalent to Definition \ref{primeideal}. Also, as in the classical case, any maximal ideal of a hyperring $R$ is prime. 
\end{rmk}

\begin{rmk}
When $R$ is a hyperfield, $R$ has a unique prime ideal, namely $\{0_R\}$.
\end{rmk}

Let $R$ be a hyperring and $\mathfrak{p}$ be a prime ideal of $R$. The localization $R_\mathfrak{p}$ of $R$ at $\mathfrak{p}$ is a hyperring with the underlying set:
\[
R_\mathfrak{p}:=(R \times S)/ \sim,  
\]
where $\sim$ is an equivalence relation on $(R \times S)$ such that $(r,a) \sim (r_1,a_1)$ if and only if there exists $c \in S:=R-\mathfrak{p}$ such that $cra_1=cr_1a$. We write $\frac{r}{a}$ for the equivalence class of $(r,a)$. Now, one imposes the following hyperaddition and multiplication:
\[
\frac{r}{a}+\frac{r'}{a'}:=\{\frac{c}{aa'} \mid c \in ar'+a'r\}, \quad \frac{r}{a}\cdot\frac{r'}{a'}:=\frac{rr'}{aa'}. 
\]
Equipped with these two operations, $R_\mathfrak{p}$ becomes a hyperring. Furthermore, we have a canonical map: $S^{-1}:R \to R_\mathfrak{p}$ sending $r$ to $\frac{r}{1}$ and this homomorphism is injective if $R$ does not have any zero-divisor. Note that in fact the localization procedure can be done at any multiplicative subset $S$ of $R$ and also satisfies the universal property as in the case of commutative rings. 

\begin{mydef}
Let $R$ be a hyperring and $X=\Spec R$ be the set of all prime ideals of $R$. Then one imposes Zariski topology on $X$ by declaring that closed sets are of the form $V(I):=\{\mathfrak{p} \in X \mid I\subseteq \mathfrak{p}\}$ for some ideal $I$ of $R$. 
\end{mydef}

\begin{rmk}
In \cite{jun2015algebraic}, it is proven that, when $A$ is a $k$-algebra and $R=A/k^\times$, the prime spectra $\Spec A$ and $\Spec R$ are homeomorphic. One can easily confirm the set bijection. In fact, we have that $\Spec A=\Hom(A,\mathbb{K})$, $\Spec R=\Hom(A/k^\times,\mathbb{K})$, and $\Hom(A,\mathbb{K})=\Hom(A/k^\times,\mathbb{K})$. 
\end{rmk}

Let $R$ be a \emph{hyperdomain}, i.e., $R$ is a hyperring without multiplicative zero-divisors. In this case, we can construct a structure sheaf (of hyperrings) for the topological space $X=\Spec R$. Indeed, we define the hyperring of sections over an open subset $U\subseteq X$ as follows:
\[
\mathcal{O}_X(U):=\{s:U\to \bigsqcup_{\mathfrak{p} \in U}R_\mathfrak{p},\textrm{ such that $s$ is locally representable by some element $\frac{a}{f}$. } \}
\]

Then one has the following:

\begin{mytheorem}\cite{jun2015algebraic}\label{equivalence}
Let $R$ be a hyperdomain and $X=\Spec R$, then the hyperring of global sections $\Gamma(X,\mathcal{O}_X)$ is isomorphic to $R$. Furthermore, for each $x=\mathfrak{p} \in X$, the stalk $\mathcal{O}_{X,x}$ exists and is isomorphic to $R_\mathfrak{p}$.  
\end{mytheorem}

One can directly generalize the definition of locally ringed spaces to define locally hyperringed spaces as follows:

\begin{mydef}\label{locallyhyperringed}
\begin{enumerate}
\item 
A \emph{locally hyperringed space} $(X,\mathcal{O}_X)$ is a topological space $X$ together with a sheaf $\mathcal{O}_X$ of hyperrings (the structure sheaf of $X$) such that the stalk $\mathcal{O}_{X,x}$ exists for each $x \in X$ and contains a unique maximal ideal.
\item
Let $(X,\mathcal{O}_X)$ and $(Y,\mathcal{O}_Y)$ be locally hyperringed spaces. A \emph{morphism} 
\[
(f,f^\#):(X,\mathcal{O}_X) \to (Y,\mathcal{O}_Y)
\] 
of locally hyperringed spaces is a pair of a continuous map $f:X \to Y$ and a morphism $f^\#:\mathcal{O}_Y \to f_*\mathcal{O}_X$ of sheaves of hyperrings such that for each $x \in X$, the induced map $f^\#_x:\mathcal{O}_{Y,f(x)} \to \mathcal{O}_{X,x}$ is local. In other words, the inverse image of a unique maximal ideal of $\mathcal{O}_{X,x}$ is a unique maximal ideal of $\mathcal{O}_{Y,f(x)}$.  
\item 
An \emph{integral hyperring scheme} is a locally hyperringed space which is locally isomorphic to the spectrum of a hyperdomain. 
\end{enumerate}
\end{mydef}

\begin{rmk}
We remark that for a sheaf $\mathcal{F}$ of hyperrings on a topological space $X$, the stalk $\mathcal{F}_x$ at a point $x \in X$ may not exist. In Definition \ref{locallyhyperringed}, a locally hyperringed space $(X,\mathcal{O}_X)$ assumes the existence of the stalk $\mathcal{O}_{X,x}$ at each $x \in X$. 
\end{rmk}

\begin{pro}\label{cominclusion}
The inclusion functor $i$, from the category $\mathcal{C}$ of commutative rings to the category $\mathcal{D}$ of hyperrings, is fully faithful. 
\end{pro}
\begin{proof}
Let $A$, $B$ be commutative rings. We have to show that $\Hom_\mathcal{C}(A,B)=\Hom_\mathcal{D}(i(A),i(B))$. Suppose that $i(f)=i(g)$ for $f,g \in \Hom_\mathcal{C}(A,B)$. In particular, this means that $f=g$ as functions and hence $f=g \in \Hom_\mathcal{C}(A,B)$. Also, for any $h\in  \Hom_\mathcal{D}(i(A),i(B))$, since $i(A)$ and $i(B)$ are commutative rings, the condition $h(a+b)\subseteq h(a)+h(b)$ for all $a,b \in i(A)$ simply means that $h(a+b)=h(a)+h(b)$ and hence $h \in \Hom_\mathcal{C}(A,B)$. 
\end{proof}

\begin{pro}\label{stalklemma}
Let $X$ be a scheme considered as an object in the category of locally hyperringed spaces. For each $x \in X$, the stalk $\mathcal{O}_{X,x}$ exists and is same as the stalk taken by considering $X$ in the category of locally ringed spaces. 
\end{pro}
\begin{proof}
We may assume that $X$ is affine, i.e., $X=\Spec A$ for some commutative ring $A$. Let $x=\mathfrak{p} \in X$. One can easily check that in this case $A_\mathfrak{p}$ satisfies the universal property of the stalk of $\mathcal{O}_X$ at $x$.  
\end{proof}

\begin{pro}\label{inclusionf}
The inclusion functor $i$, from the category $\mathfrak{Schm}$ of schemes to the category $\mathfrak{Lhs}$ of locally hyperringed spaces, is fully faithful. 
\end{pro}
\begin{proof}
Let $X$ and $Y$ be schemes. We have to show that $\Hom_\mathfrak{Schm}(X,Y)=\Hom_\mathfrak{Lhs}(i(X),i(Y))$. Since any morphism of locally ringed spaces is indeed a morphism of locally hyperringed spaces, we have that 
\[
\Hom_\mathfrak{Schm}(X,Y)\subseteq\Hom_\mathfrak{Lhs}(i(X),i(Y))
\]
Let $(h,h^\#):i(X) \to i(Y)$ be a morphism of locally hyperringed spaces. Similar to Proposition \ref{cominclusion}, one can easily see that $(h,h^\#)$ is indeed a morphism of locally ringed spaces. This completes the proof. 
\end{proof}

%\begin{rmk}
%\begin{enumerate}
%\item 
%In general, (co)limits do not exist in the category of hyperrings and hence one should also include the existence of stalks as a part of Definition \ref{locallyhyperringed} in contrast to the case of locally ringed spaces. Nonetheless, for any hyperring $R$, the topological space $X=\Spec R$ is a spectral space. It would be interesting to see whether the result of M.~Hochster \cite{hochster1969prime} can be generalized to hyperrings. i.e., to investigate whether any spectral space arises as a prime spectrum of a hyperring which is not a ring. 
%\item
%It is proven in \cite{jun2015algebraic} that when $R$ is a hyperdomain then $\Spec R$ is a locally hyperringed space with the stalk $R_\mathfrak{p}$ at each $\mathfrak{p} \in \Spec R$. 
%\end{enumerate}
%\end{rmk}

Note that Theorem \ref{equivalence} can not be generalized to the category of hyperrings. For instance, the following example shows that one does not have an equivalence between the opposite category of hyperrings and the category of affine hyperring schemes. 

\begin{myeg}\cite[Example 4.24]{jun2015algebraic}
Let $R$ be the quotient hyperring $(\mathbb{Q}\oplus \mathbb{Q})/G$, where $G$ is the subgroup of ($\mathbb{Q}\oplus \mathbb{Q})^\times$ consisting of $(1,1)$ and $(-1,-1)$. Let $X=\Spec R$. Then we have
\[
\Gamma(X,\mathcal{O}_X)=(\mathbb{Q}/N) \oplus (\mathbb{Q}/N), 
\] 
where $N=\{1,-1\}$ is a subgroup of $\mathbb{Q}^\times$. One can easily check that in this case $R$ is not isomorphic to $\Gamma(X,\mathcal{O}_X)$.  
\end{myeg}

At the moment, we only have theory of integral hyperring schemes which only generalize integral schemes. Nonetheless, we have the following theorem which will be used in the sequel. For the general terminology for hyperring schemes which generalizes the classical concepts, we refer the readers to \cite{jun2015algebraic} or \cite[\S 4]{jaiungthesis}.  

\begin{mythm}\label{mainlemma}
Let $H$ be a hyperfield and $k$ be a field. Suppose that we have a fixed homomorphism $s:k \to H$ of hyperfields. Let $X$ be a scheme over a field $k$. Then to give a morphism $f:\Spec H \to X$ over $k$ of locally hyperringed spaces is equivalent to give a point $x \in X$ and a homomorphism $\tilde{s}:k(x) \to H$ of hyperrings such that $s=\tilde{s}\circ \rho$, where $k(x)$ is the residue field at $x$ and $\rho:k \to k(x)$ is the canonical homomorphism. 
\end{mythm}
\begin{proof}
The proof is similar to the classical proof, however, we include the proof for the sake of completeness. Let $\Spec H=\{y\}$. First, suppose that $(f,f^\#):\Spec H \to X$ sending $y$ to $x$ is a morphism of locally hyperringed spaces. By taking stalks, we have 
\[
f^\#_x:\mathcal{O}_{X,x} \to H. 
\]
Note that even though the stalk $\mathcal{O}_{X,x}$ is taken in the category of locally hyperringed spaces, it is same as being taken in the category of locally ringed spaces thanks to Proposition \ref{stalklemma}. Now, since $f_x^\#$ is local, we have $(f_x^\#)^{-1}(y)=(f_x^\#)^{-1}(\{0\})=\mathfrak{m}_x$, where $\mathfrak{m}_x$ is a unique maximal ideal of $\mathcal{O}_{X,x}$. In other words, $\mathfrak{m}_x = \ker f_x^\#$ and hence $f^\#_x$ induces the following homomorphism of hyperrings:
\[
\varphi_x:\mathcal{O}_{X,x}/\mathfrak{m}_x=k(x) \to H, \quad [a] \mapsto f_x^\#(a),
\]
where $[a]$ is the equivalence class of $a \in \mathcal{O}_{X,x}$ in $k(x)$. This shows that $(f,f^\#)$ gives a point $x \in X$ and a homomorphism of hyperrings $\varphi_x:\mathcal{O}_{x,x} \to H$. This is clearly compatible with $s:k \to H$ and $\rho:k\to k(x)$.

Conversely, suppose that $x \in X$ and $\varphi_x:k(x) \to H$ such that $\varphi_x\circ \rho=s$ are given. We define $f:\Spec H=\{y\} \to X$ sending $y$ to $x$. Then clearly $f$ is continuous. Next, we define a morphism $f^\#:\mathcal{O}_X \to f_*\mathcal{O}_{\Spec H}$ of sheaves of hyperrings. Notice that
\[
\mathcal{O}_{\Spec H}(f^{-1}(U)) =\left\{ \begin{array}{ll}
H & \textrm{if $x \in U$}\\
0& \textrm{if $x \not \in U$},
\end{array} \right.
\]
Hence, for each open subset $U$ of $X$, we define the following:
\[
f^\#(U):=\left\{ \begin{array}{ll}
\varphi_x\circ \pi \circ \pi_{U,x} & \textrm{if $x \in U$}\\
0& \textrm{if $x \not \in U$},
\end{array} \right.
\]
where $\varphi_x$ is given, $\pi:\mathcal{O}_{X,x} \to k(x)$ is a canonical projection, and $\pi_{U,x}$ is a canonical homomorphism from $\mathcal{O}_X(U)$ to $\mathcal{O}_{X,x}$. Then $f^\#$ is a morphism of sheaves of hyperrings. Indeed, clearly $f^\#(U)$ is a homomorphism of hyperrings for each open subset $U$ of $X$ and hence we only have to check the compatibility condition. Suppose that $V \subseteq U \subseteq X$. There are three cases; the first case is when $x \not \in U$. In this case, both $f^\#(U)$ and $f^\#(V)$ are zero and hence there is nothing to prove. The second  case is when $x \in U \cap V^c$. Since $\mathcal{O}_{\Spec H}(f^{-1}(V))=0$ and $f^\#(V)=0$, in this case the compatibility is clear. The only nontrivial case is when $x \in V$. In this case, we have $\mathcal{O}_{\Spec H}(f^{-1}(U))=\mathcal{O}_{\Spec H}(f^{-1}(V))=H$ and the restriction map $\pi_{f^{-1}(U),f^{-1}(V)}:\mathcal{O}_{\Spec H}(f^{-1}(U)) \to \mathcal{O}_{\Spec H}(f^{-1}(V))$ is just an identity map. We first claim that the following diagram commutes. 
\begin{equation}
\begin{gathered} 
\xymatrix{\ar @{} [dr]
\mathcal{O}_X(U)  \ar[r]^{\pi_{U,V}} \ar[d]_{f^\#(U)} & \mathcal{O}_X(V) \ar[d]^{f^\#(V)}   \\
 H \ar[r]^{id} & H  }
\end{gathered} 
\end{equation}
In fact, we have $f^\#(U)=\varphi_x\circ \pi \circ \pi_{U,x}$. But, since $x \in V \subseteq U$, we have $\pi_{U,x}=\pi_{V,x}\circ \pi_{U,V}$ and hence $f^\#(U)=\varphi_x\circ \pi\circ \pi_{V,x}\circ \pi_{U,V}=f^\#(V)\circ \pi_{U,V}$. This proves that the diagram commutes. It only remains to show that $f^\#_{x}:\mathcal{O}_{X,x} \to \mathcal{O}_{\Spec H,y}$ is a local homomorphism of local hyperrings, i.e., $(f^\#_x)^{-1}(y)=\mathfrak{m}_x$. For this, we may assume that $X$ is affine. Let $X=\Spec A$ and $x$ be a prime ideal $\mathfrak{p}$. Since $\mathcal{O}_{\Spec H,y}=\Frac(H)=H$ (thanks to Theorem \ref{equivalence}), by taking global sections and stalks, we have the following commutative diagram:
\begin{equation}\label{equation}
\begin{gathered} 
\xymatrix{\ar @{} [dr]
A  \ar[r]^{\pi_{X,x}} \ar[d]_{f^\#(X)} & A_\mathfrak{p} \ar[d]^{f^\#_y}   \\
 H \ar[r]^{id} & H  }
 \end{gathered} 
\end{equation}
Notice that $y=\{0_H\} \subseteq H$ and $(f^\#(X))^{-1}(y)=\mathfrak{p}$ (from the assumption). Furthermore, we have $\pi_{X,x}(\mathfrak{p})=\mathfrak{m}_x$. It follows from the commutative diagram \eqref{equation} that $(f^\#_y)^{-1}(y)=\mathfrak{m}_x$.

One can easily check that the above two constructions are inverse to each other and hence induces the desired one-to-one correspondence. 
\end{proof}

\begin{rmk}
In Theorem \ref{mainlemma}, we considered a scheme over a field $k$, however, one can easily observe that the same argument is still valid when one replaces $k$ with $\mathbb{Z}$, the ring of integers. 
\end{rmk}

\begin{rmk}
One may notice that the above theorem is a generalization of a classical result; when $H$ is a field, then the only structure morphism $\varphi:k \to H$ is an inclusion (or a canonical map when $k$ is $\mathbb{Z}$) and to give a morphism of $\Spec H$ to $X$ is equivalent to give a point $x \in X$ and an inclusion map $k(x) \to H$ which is compatible with the given structure morphism. 
\end{rmk}

When we specialize a hyperfield $H$ to be the Krasner hyperfield $\mathbb{K}$, the tropical hyperfield $\mathbb{T}$, and the hyperfield of signs $\mathbb{S}$ , we have the following. 
 
\begin{cor}\label{krasner}
Let $\mathbb{K}$ be the Krasner hyperfield. Suppose that $X$ is a scheme over $\mathbb{Z}$. To give a morphism $f:\Spec \mathbb{K}=\{y\} \to X$ is the same thing as to give a point $f(y)=x \in X$ and a homomorphism $k(x) \to \mathbb{K}$. 
\end{cor}
%\begin{proof}
%For any field $K$, since $\Spec K$ is a one point set and $\Spec K =\Hom(K,\mathbb{K})$, there is only one homomorphism from $K$ to $\mathbb{K}$. In particular, there exists only one homomorphism from $k(x)$ to $\mathbb{K}$ at each $x \in X$. Therefore, a morphism $f:\Spec \mathbb{K}=\{y\} \to X$ is uniquely determined by a point $f(y )=x \in X$. In other words, the underlying set $|X|$ of a scheme $X$ is in a one-to-one correspondence with the set $\Hom_{\mathfrak{Lhs}}(\Spec \mathbb{K},X)$.	
%\end{proof}

\begin{rmk}\label{krasnerrmk}
For any field $K$, since $\Spec K$ is a one point set and $\Spec K =\Hom(K,\mathbb{K})$, there is only one homomorphism from $K$ to $\mathbb{K}$. In particular, there exists only one homomorphism from $k(x)$ to $\mathbb{K}$ at each $x \in X$. Therefore, a morphism $f:\Spec \mathbb{K}=\{y\} \to X$ is uniquely determined by a point $f(y )=x \in X$. In other words, the underlying set $|X|$ of a scheme $X$ is in a one-to-one correspondence with the set $\Hom_{\mathfrak{Lhs}}(\Spec \mathbb{K},X)$. This perspective will be investigated in \S \ref{schemetheory}. 
\end{rmk}

\begin{cor}\label{analytification}
Let $\mathbb{T}$ be the tropical hyperfield. Let $k$ be a non-Archimedean valued field with a valuation $\nu:k\to \mathbb{T}$ and $X$ be a scheme over $k$. Then to give a morphism $f:\Spec \mathbb{T}=\{y\} \to X$ is the same thing as to give a point $f(y)=x \in X$ and an extension of $\nu$ from $k$ to $k(x)$. 
\end{cor}

\begin{rmk}
This perspective will be studied in \S \ref{berkovich} in connection to the Berkovich analytification $X^{\textrm{an}}$ of an algebraic variety $X$ over a complete non-Archimedean valued field $k$.
\end{rmk}

\begin{cor}
Let $\mathbb{S}$ be the hyperfield of signs. Let $k$ be a field together with a homomorphism $s:k \to \mathbb{S}$ and $X$ be a scheme over $k$. Then to give a morphism $f:\Spec \mathbb{S}=\{y\} \to X$ is the same thing as to give a point $f(y)=x \in X$ and a homomorphism from $k(x)$ to $\mathbb{S}$ whose restriction to $k$ is $s$, that is, an ordering of $k(x)$ which extends that of $k$. 
\end{cor}

\begin{rmk}
This perspective will be investigated in \S \ref{spaceoforderings} in connection to the real scheme $X_r$ associated to a scheme $X$. 
\end{rmk}

We also have the following proposition which shows that our notion of locally hyperringed spaces is an appropriate generalization from affine schemes to non-affine schemes. In what follows, by $\mathfrak{Lhs}$, we always mean the category of locally hyperringed spaces. 

\begin{pro}\label{equivaelnceofcategories}
Let $H$ be a hyperfield and $X=\Spec A$ be an affine scheme. Then, we have the following identification of sets:
\[
\Hom(A,H)=\Hom_{\mathfrak{Lhs}}(\Spec H, X).
\]  
\end{pro}
\begin{proof}
Let $f:A \to H$ be a homomorphism of hyperrings. Then this determines the point $\mathfrak{p}:=\ker(f) \in \Spec A$ and also $f$ factors through $\tilde{f}:A/\mathfrak{p} \to H$ and hence induces a homomorphism 
\[
\Frac(A/\mathfrak{p})=k(p) \longrightarrow H.
\]
It follows from Theorem \ref{mainlemma} that this determines a unique element in $\Hom_{\mathfrak{Lhs}}(\Spec H, X)$. Conversely, any given morphism $g:\Spec H \to X$ induces a homomorphism $\Gamma(X)=A \to \Gamma(\Spec H)=H$ (thanks to Theorem \ref{equivalence}). These are clearly inverses to each other. 
\end{proof}

\begin{mydef}
Let $X$ be a scheme and $H$ be a hyperfield.
\begin{enumerate}
\item 
We let $X(H)$ be the set of $H$-rational points of $X$, i.e., 
\[
X(H):=\Hom_{\mathfrak{Lhs}}(\Spec H, X).
\]
Whenever there is no possible confusion, we will simply write $X(H)=\Hom(\Spec H, X)$. 
\item
Let $X$ be a scheme over a field $k$ and $H$ be a hyperfield with a fixed homomorphism $\varphi:k \to H$. We let $X(H):=\Hom_k(\Spec H, X)$, i.e., the set of morphisms of locally hyperringed spaces from $\Spec H$ to $X$ which are compatible with $\varphi^\#:\Spec H \to \Spec k$. 
\end{enumerate}
\end{mydef}

\begin{rmk}
It follows from Proposition \ref{equivaelnceofcategories} that, when $X=\Spec A$ and $H$ is a hyperfield, we have $X(H)=\Hom(A,H)$ (as sets). 
\end{rmk}

\subsection{Fine topology on sets of rational points}\label{finetopology}
Let $X$ be a scheme over a scheme $S$. Then, in general, $X(S)$ is not equipped with any topology. In this section, we follow the idea of Lorscheid and Salgado in \cite{lorscheid2016remark} (or Lorscheid in \cite{lorscheid2015scheme}) to impose the \emph{fine topology} on sets of rational points of schemes over hyperfields with arbitrary topology. In what follows, a hyperfield with a topology is always assumed to be a topological hyperfield defined below. 

\begin{mydef}
A hyperfield $H$ is said to be a \emph{topological hyperfield} if $H$ is equipped with topology, satisfying the following conditions:
\begin{enumerate}
	\item 
	The multiplication map $H \times H \to H$, where $H \times H$ is equipped with product topology, is continuous. 
	\item 	
	$H^\times=H-\{0\}$ is open, and the inversion map $i:H^\times \to H^\times$ is continuous. 
\end{enumerate}
\end{mydef}

\begin{myeg}
One can easily check the following. 
\begin{enumerate}
	\item 
The Krasner hyperfield $\mathbb{K}$ with the topology $\{\emptyset, \{1\}, \{0,1\}\}$ is clearly a topological hyperfield. 
\item 
The tropical hyperfield $\mathbb{T}$ with the Euclidean topology is a topological hyperfield. 
\item 
The hyperfield of signs $\mathbb{S}$ with the topology $\{\emptyset, \{1\},\{-1\}, \{-1,1\}, \mathbb{S}\}$ is a topological hyperfield.
\end{enumerate}	
\end{myeg}

\begin{rmk}
We remark that, in \cite{lorscheid2015scheme}, Lorscheid implements the fine topology to ordered blue schemes to recast Berkovich analytification topologically. Also, a similar idea is considered by J.~Giansiracusa and N.~Giansiracusa in \cite{giansiracusa2014universal} based on their previous work \cite{noah} for the same purpose as Lorscheid, but with tropical schemes. To be a bit more precise, Giansiracusa-Giansiracusa show in \cite{giansiracusa2014universal} that the Berkovich analytification $X^{\textrm{an}}$ is homeomorphic to the space $\Trop(X)(\mathbb{R}_{max})$ of rational points of the universal scheme-theoretic tropicalization $\Trop(X)$ over the tropical semifield $\mathbb{R}_{max}$ endowed with the strong Zariski topology (\cite[Definition 3.4.1]{giansiracusa2014universal}). 
\end{rmk}

The recipe is as follows. Let $X$ be a scheme over a field $k$ (or over $\mathbb{Z}$) and $H$ be a hyperfield which is equipped with a topology $\mathcal{T}$. Suppose that a structural map $\varphi:k \to H$ is fixed. First, we consider when $X$ is an affine, i.e., $X=\Spec A$. In this case, $X(H)=\Hom_k(\Spec H, X)=\Hom_k(A,H)$ and hence we consider the following identification:
\begin{equation}
X(H)=\Hom_k(A,H) \subseteq \prod_{a \in A}H^{(a)}. 
\end{equation}
We give the product topology on $\prod_{a \in A}H^{(a)}$ by using the topology $\mathcal{T}$ on $H$ and then impose the subspace topology $\mathcal{T}_p$ on $X(H)$. This is called the \emph{affine topology}. We note that the topology $\mathcal{T}_p$ is the coarsest topology on $\Hom_k(A,H)$ such that for each $a \in A$, the evaluation map 
\[
ev_a:\Hom_k(A,H) \to H,\quad f \mapsto f(a) 
\]
is continuous. In this case, one can easily check that this topology is functorial in both $A$ and $H$. 

Next, consider the case when $X$ is a scheme over a field $k$ (or over $\mathbb{Z}$), with a structural map $\varphi:k\to H$. The \emph{fine topology} $\mathcal{T}_F$ on $X(H)$ is the finest topology such that for any $k$-morphism $f_Y:Y \to X$ from an affine $k$-scheme $Y$ to $X$, the induced map 
\[
f_Y(H):Y(H) \to X(H)
\] 
is continuous, where $Y(H)$ is equipped with the affine topology.

The following in the classical setting is proven in \cite{lorscheid2016remark}, and one may apply a similar argument in \cite[Theorem A]{lorscheid2016remark} to prove the following:

\begin{pro}\label{affinetopology}
With the above notations, if $X$ is an affine scheme and $H$ is a hyperfield equipped with topology, then the affine topology and the fine topology agree on $X(H)$. 
\end{pro}

The fine topology is functorial in the following sense. 

\begin{pro}\label{inducedlemma}
Let $f:Y \to X$ be a morphism of schemes and $H$ be a topological hyperfield. Then the induced map, $f(H):Y(H) \to X(H)$ is continuous, where $Y(H)$ and $X(H)$ are equipped with the fine topology. 
\end{pro}
\begin{proof}
The same argument as in \cite[Proposition 2.1]{lorscheid2016remark} shows the result. 
\end{proof}

\begin{pro}\label{openembedding}
Let $X$ be a scheme and $H$ be a topological hyperfield. If $\{U_i\}$ is an affine open covering of $X$. Then $\{U_i(H)\}$ is an open covering of $X(H)$. 
\end{pro}
\begin{proof}
The standard argument, as in the proof of \cite[Theorem B]{lorscheid2016remark}, reduces our proposition to the case when $X$ is affine, say $X=\Spec A$. We may also assume that $U_i$ is a basic open subset $D(f_i)$ of $X$ for some $f_i \in A$. In this case, since the affine topology agrees with the fine topology, we may further assume that $X(H)$ and $U_i(H)$ are equipped with the affine topology. 

Now, for each $f \in A$, let $\alpha:A \to A_f$ be the localization. One can easily see that following is an injection:
\begin{equation}\label{eq: loc}
\tilde{\alpha}: D(f)(H)=\Hom(A_f, H) \to X(H)=\Hom(A,H), \quad \varphi \mapsto \varphi \circ \alpha.
\end{equation}
Furthermore, we have
\begin{equation}\label{eq: open cover}
\textrm{Img}(\tilde{\alpha}) = \{\psi \in \Hom(A,H) \mid \psi(f) \neq 0\}.
\end{equation}
From Proposition \ref{inducedlemma}, $\tilde{\alpha}$ is continuous. We claim that $\tilde{\alpha}$ is an open map. In fact, let $Z$ be an open subset of $\Hom(A_f,H)$ (with the affine topology). We may assume that
\[
Z=\{\varphi:A_f \to H \mid \varphi(\frac{a}{f^n}) \in W\}
\]
for some fixed $\frac{a}{f^n} \in A_f$ and a fixed open subset $W$ of $H$. We have the following from \eqref{eq: loc} and \eqref{eq: open cover}:
\[
\tilde{\alpha}(Z)= \{\psi:A \to H \mid \psi(f) \neq 0, \psi(a)\psi(f)^{-n} \in W\}.
\]
Let's consider $T:=\prod_{a \in A}H^{(a)}$ as the set of functions from $A$ to $H$ (not necessarily homomorphisms). Then, the following is an open subset of $T$
\[
T_W=\{g \in T \mid g(f) \neq 0, ~~g(a)g(f)^{-n} \in W\}
\]
since the multiplication and inversion of $H$ are continuous and $H^\times$ is an open subset of $H$. Now, we have that
\[
\tilde{\alpha}(Z)=T_W \cap \Hom(A,H),
\]
and hence $\tilde{\alpha}(Z)$ is an open subset of $\Hom(A,H)$, showing that $\tilde{\alpha}$ is an open embedding. 

Finally, if $\{U(f_i)\}$ is an open cover of $\Spec A$, we can find $n \in \mathbb{N}$ such that $\sum_{i=1}^n f_i =1$. It follows that for any $\psi \in \Hom(A,H)$, one has:
\[
1=\psi(1)=\psi(\sum_{i=1}^nf_i) \in \sum_{i=1}^n \psi(f_i). 
\]
In particular, for each $\psi \in \Hom(A,H)$, $\psi(f_i) \neq 0$ for some $i=1,\dots,n$. It follows from the description \eqref{eq: open cover} that $\{U(f_i)(H)\}$ is an open cover of $X(H)$. 
\end{proof}

\section{Geometry of hyperfields in a view of classical scheme theory}\label{schemetheory}

Let $X$ be a scheme over $\mathbb{Z}$ and $\mathbb{K}$ be the Krasner hyperfield. In \cite{con3}, Connes and Consani showed that 
\begin{equation}\label{sets}
X=X(\mathbb{K}) \textrm{ (as sets)}.
\end{equation}
One can see that the bijection in \eqref{sets} easily follows from Theorem \ref{mainlemma}. We enrich the bijection \eqref{sets} to a homeomorphism in a natural way. To this end, we impose the topology $\mathcal{T}$ on $\mathbb{K}=\{0,1\}$ in such a way that the open sets are $\emptyset$, $\{1\}$, and $\{0,1\}$. Then we have the following:

\begin{pro}\label{schemehomeo}
Let $X=\Spec A$ be an affine scheme over $\mathbb{Z}$. Then $X(\mathbb{K})$ (equipped with the fine topology) is homeomorphic to $X$ (equipped with the Zariski topology). 
\end{pro}
\begin{proof}
Let $\varphi:X \to X(\mathbb{K})=\Hom(A,\mathbb{K})$ be the set bijection from Theorem \ref{mainlemma}, i.e., 
\[
\varphi:X \to X(\mathbb{K})=\Hom(A,\mathbb{K}), \quad \mathfrak{p} \mapsto \varphi(\mathfrak{p}):=a_\mathfrak{p},
\] 
where $a_\mathfrak{p}(x)=0 $ if and only if $x \in \mathfrak{p}$.  For $a \in A$, let $D(a)$ be the basic open subset of $X$, i.e., $D(a)=\{\mathfrak{p} \in X \mid a \not \in \mathfrak{p}\}$. Then, we have
\[
\varphi(D(a))=\{f \in \Hom(A,\mathbb{K}) \mid a \notin \ker(f)\}=\Hom(A,\mathbb{K})\bigcap \left(\prod_{r \in A} U_r\right),
\]
where $U_a=\{1\}$ and $U_r=\mathbb{K}$ for all $r \neq a \in A$. Clearly $\prod_{r \in A} U_r$ is an open subset of $\prod_{r \in A}\mathbb{K}^{(r)}$ and hence $\Hom(A,\mathbb{K})\bigcap \left(\prod_{r \in A} U_r\right)$ is an open subset of $\Hom(A,\mathbb{K})$.

Conversely, suppose that $U$ is an open subset of $\Hom(A,\mathbb{K})$. We may assume that $U=\prod_{a \in A}U_a$ such that $U_a=\mathbb{K}$ for all but finitely many $a_1,...,a_n$, where $U_{a_i}=\{1\}$. One can easily check that 
\[
\varphi^{-1}(U)=\bigcap_{i=1}^n D(a_i).
\]  
This proves that $\varphi$ is a homeomorphism. 
\end{proof}

Indeed, Proposition \ref{schemehomeo} can be generalized to any scheme over $\mathbb{Z}$ as follows. 

\begin{pro}\label{SPECrepresentable}
Let $X$ be a scheme over $\mathbb{Z}$. Then $X(\mathbb{K})$, equipped with the fine topology, is homeomorphic to $X$. 
\end{pro}
\begin{proof}
Let $i:X\to X(\mathbb{K})$ be the set-bijection described in Corollary \ref{krasner}. Let's fix an affine open covering $\{U_i\}$ of $X$, where $U_i=\Spec A_i$. Consider the restriction $r_i:=i|_{U_i}:U_i \to U_i(\mathbb{K})$ of $i$ on each $U_i$. It follows from Proposition \ref{openembedding} that each $U_i(\mathbb{K})$ is an open subset of $X(\mathbb{K})$ and $\{U_i(\mathbb{K})\}$ is an open covering of $X(\mathbb{K})$. Now, it follows from Proposition \ref{schemehomeo} that each $r_i$ is a homeomorphism and the desired result follows.
\end{proof}

\begin{cor}
Let $\mathcal{F}$ be the functor from the category $\mathfrak{Schm}$ of schemes to the category $\mathfrak{Top}$ of topological spaces sending a scheme $X$ to its underlying topological space $|X|$. Then $\mathcal{F}$ is isomorphic to the functor $\Hom(\Spec \mathbb{K},-)$. In particular, by considering the affine case, the functor $\Spec$, from the category of commutative rings to $\mathfrak{Top}$, is isomorphic to $\Hom(-,\mathbb{K})$. 
\end{cor}
\begin{proof}
For the notational convenience, we let $\mathcal{G}:=\Hom(\Spec \mathbb{K},-)$. For each scheme $X$, we let $\eta_X:\mathcal{F}(X)=|X| \to \mathcal{G}(X)=X(\mathbb{K})$ be the homeomorphism in Proposition \ref{SPECrepresentable}. Then, for each $\mathfrak{p} \in |X|$, we have $\eta_X(\mathfrak{p}):\Spec \mathbb{K} \to X$ such that the image of $\eta_X(\mathfrak{p})$ is $\mathfrak{p} \in X$. In particular, for a morphism of schemes $f:X\to Y$, one can easily see that the following diagram commutes:
\[
\begin{tikzcd}[row sep=large, column sep=1.5cm]
\mathcal{F}(X)\arrow{r}{\eta_X}\arrow{d}[swap]{\mathcal{F}(f)}
& \mathcal{G}(X) \arrow{d}{\mathcal{G}(f)} \\
\mathcal{F}(Y) \arrow{r}[swap]{\eta_Y} 
& \mathcal{G}(Y)
\end{tikzcd}
\]
This proves that the functors $\mathcal{F}$ and $\mathcal{G}$ are isomorphic. 
\end{proof}

%\begin{rmk}\label{zariskitop}
%Let $X=\Spec A$ be an affine scheme. One can rephrase the homeomorphism in Proposition \ref{schemehomeo} as follows: let $I$ be an ideal of $A$, we let $V(I)$ be the Zariski closed subset of $X$, and 
%\begin{equation}\label{KrasnerZariski}
%V_{\mathbb{K}}(I):=\{f \in X(\mathbb{K}) \mid f(x)=0, \quad \forall x\in I\}.
%\end{equation}
%Then we have $\varphi(V(I))=V_{\mathbb{K}}(I)$. Indeed, if $\mathfrak{p} \in V(I)$, then $I \subseteq \mathfrak{p}$. This implies that for each $x \in I$, $\varphi(\mathfrak{p})(x)=a_\mathfrak{p}(x)=0$ and hence $a_\mathfrak{p} \in V_\mathbb{K}(I)$. Conversely, if $a_\mathfrak{p} \in V_\mathbb{K}(I)$, then for each $x \in I$, we have $a_\mathfrak{p}(x)=0$. It follows that $x \in \mathfrak{p}$ and hence $I \subseteq \mathfrak{p}$ and $p \in V(I)$. 
%\end{rmk}

\section{Geometry of hyperfields in a view of Berkovich theory}\label{berkovich}

In this section, we study the Berkovich analytification of an algebraic variety in terms of the tropical hyperfield $\mathbb{T}$. We also consider a possible connection to tropical geometry and the Berkovich analytification of affine algebraic group schemes. We note that Berkovich used the multiplicative notation (with $\mathbb{R}_{\geq 0}$), however, we will use the additive notation to be compatible with the additive notation of $\mathbb{T}$ and this makes no difference. 

\subsection{Analytification is representable}
We generalize and prove the remark that Baker and Bowler made (for the affine case) in \cite[Example 5.4]{baker2016matroids}; Berkovich analytification functor is isomorphic to the functor $\Hom(\Spec \mathbb{T},-)$. In what follows, we always assume that $k$ is a complete non-Archimedean valued field and let $\mathbb{T}$ be the tropical hyperfield. As we mentioned earlier, we will use the additive notation and all valuations will be assumed to be complete and non-Archimedean unless otherwise stated. We also use the terms \emph{multiplicative seminorm} and \emph{semivaluation} interchangeably.

Let's first see the affine case. For a normed algebra $(\mathcal{A},|-|_\mathcal{A})$ over a field $k$ with a valuation $\nu$, we define the following notation:
\begin{equation}\label{bounded}
\Hom_{b,k}(\mathcal{A},\mathbb{T}):=\{f:\mathcal{A} \to \mathbb{T} \mid \exists C_f\in \mathbb{R}\textrm{ such that } f(x) \leq C_f+|x|_\mathcal{A}\textrm{ for all }x \in \mathcal{A}  \textrm{ and } f\mid_{k}=\nu\}.
\end{equation}

\begin{pro}\label{extension them}
Let $(k,\nu)$ be a valued field, $\mathcal{A}$ be a normed algebra over $k$, and $\hat{\mathcal{A}}$ be the completion of $\mathcal{A}$. Then we have the following bijection of sets:
\[
\Hom_{b,k}(\mathcal{A},\mathbb{T})=\Hom_{b,k}(\hat{\mathcal{A}},\mathbb{T}).
\]
\end{pro}
\begin{proof}
Note that $f \in \Hom_{b,k}(\mathcal{A},\mathbb{T})$ is nothing but a bounded multiplicative seminorm on $\mathcal{A}$ extending a valuation on $k$ and it is well known that any bounded multiplicative seminorm uniquely extends to a bounded multiplicative seminorm on its completion $\hat{\mathcal{A}}$. 
\end{proof}

For a normed ring $A$, we let $\Hom_b(A,\mathbb{T})$ be the set of bounded homomorphisms from $A$ to $\mathbb{T}$ as in \eqref{bounded}.

Recall that a semivaluation on a commutative ring $A$ assumes the same axioms as a valuation except that a semivaluation allows a nontrivial kernel (this is a multiplicative seminorm in the terminology of Berkovich in \cite{berkovich2012spectral}). 

\begin{lem}\label{valuationmorphism}
Let $(\Gamma,+)$ be a totally ordered abelian group and $\Gamma_{hyp}$ be the associated hyperfield (as in Example \ref{mainexample}). Let $A$ be a commutative ring. Then a semivaluation on $A$, with the value group $\Gamma$, is equivalent to a hyperring homomorphism from $A$ to $\Gamma_{hyp}$. In particular, a real semivaluation on $A$ is the same thing as a hyperring homomorphism from $A$ to $\mathbb{T}$. 
\end{lem}
\begin{proof}
The definition of a homomorphism from a commutative ring $A$ to $\Gamma_{hyp}$ is precisely the definition of a semivaluation with the value group $\Gamma$. 
\end{proof}

%\begin{rmk}
%Let $A$ be a commutative Banach ring and $\mathcal{M}(A)$ be the Berkovich spectrum of $A$. Then we have the following canonical map:
%\begin{equation}\label{kernelmap}
%\ker:\mathcal{M}(A) \longrightarrow \Spec A, \quad |-|_x \mapsto \ker(|-|_x). 
%\end{equation}
%Since $\mathbb{K}$ is the final object in the category of hyperrings, we have a unique homomorphism $\pi:\mathbb{T}\to \mathbb{K}$. Then, in terms of hyperfields, \eqref{kernelmap} can be written as follows:
%\begin{equation}\label{kermap}
%\ker: \Hom_b(A,\mathbb{T})\longrightarrow \Hom(A,\mathbb{K}), \quad f \mapsto \pi \circ f.
%\end{equation}
%For each $\mathfrak{p} \in \Spec A$, one can associate an element $|-|_{\mathfrak{p}}$ of $\mathcal{M}(A)$ which is induced by a trivial norm on $A/\mathfrak{p}$. Now, we have a section of the map $\ker$ in \eqref{kermap} as follows:
%\begin{equation}\label{section}
%s:\Spec A \longrightarrow \mathcal{M}(A), \quad \mathfrak{p}  \mapsto |-|_\mathfrak{p}.
%\end{equation}
%Notice that $\mathbb{T}$ is a hyperfield extension of $\mathbb{K}$, i.e., there exists a homomorphism $i:\mathbb{K} \to \mathbb{T}$ of hyperrings which is also an injection. Therefore, we have an inclusion $i:\mathbb{K} \to \mathbb{T}$. Then the section $s$ in \eqref{section} can be written as follows:
%\begin{equation}
%s:\Hom(A,\mathbb{K}) \longrightarrow \Hom(A,\mathbb{T}), \quad f \mapsto i \circ f. 
%\end{equation}
%\end{rmk}

Next, we prove that the Berkovich analytification $X^{\textrm{an}}$ of a scheme $X$ of finite type over $k$ is homeomorphic to $X(\mathbb{T})$ equipped with the fine topology.

\begin{lem}\label{valuationashomomorphism}
Let $A$ be a commutative ring and $\mathbb{T}$ be the tropical hyperfield. Then giving a hyperring morphism from $A$ to $\mathbb{T}$ is equivalent to giving a prime ideal $\mathfrak{p}$ of $A$ and a real valuation on the residue field $k(\mathfrak{p})$ at $\mathfrak{p}$. 
\end{lem}
\begin{proof}
Let $\varphi:A \to \mathbb{T}$ be a homomorphism of hyperrings. One can easily see that $\mathfrak{p}:=\ker (\varphi)$ is a prime ideal of $A$. Furthermore, $\varphi$ factors through $A/\mathfrak{p}$ and induces a hyperring homomorphism $\bar{\varphi}:A/\mathfrak{p} \to \mathbb{T}$. This, in turn, induces a hyperring homomorphism $\Frac(A/\mathfrak{p})=k(\mathfrak{p}) \to \mathbb{T}$ which is a real valuation on $k(\mathfrak{p})$ by Lemma \ref{valuationmorphism}.

Conversely, suppose that we have a prime ideal $\mathfrak{q}$ and a hyperring homomorphism $f:\Frac(A/\mathfrak{q}) \to \mathbb{T}$. One can easily check that this can be lifted to define a hyperring homomorphism $\hat{f}:A \to \mathbb{T}$ such that $\ker(\hat{f})=\mathfrak{q}$. 
\end{proof}

Let $k$ be a valued field. From Lemma \ref{valuationmorphism}, this is equivalent to a field $k$ with a fixed a hyperring homomorphism $\nu:k \to \mathbb{T}$ such that $\ker(\nu)=\{0\}$.

\begin{lem}\label{affinelem}
Let $k$ be a field with a valuation $\nu:k \to \mathbb{T}$. Let $A$ be a commutative $k$-algebra. Then a semivaluation on $A$ which extends $\nu$ is the same thing as a hyperring homomorphism $f:A \to \mathbb{T}$ such that $f|_k=\nu$. 
\end{lem}
\begin{proof}
This is straightforward. 
\end{proof}

Let $k$ be a valued field with a valuation $\nu:k\to \mathbb{T}$ and $A$ be a commutative $k$-algebra. We define the following set:
\[
\Hom_k(A,\mathbb{T}):=\{f \in \Hom(A,\mathbb{T}) \mid f|_k=\nu\}.
\]
Then we have the following.

\begin{pro}\label{affinecase}
Let $X=\Spec A$ be an affine scheme of finite type over a field $k$ with a valuation $\nu$. Then we have the following bijection of sets:
\begin{equation}\label{affinehomeo}
X^{an}=X(\mathbb{T}) (:=\Hom_k(A,\mathbb{T})). 
\end{equation}
Furthermore, the bijection \eqref{affinehomeo} is a homeomorphism when $X(\mathbb{T})$ is equipped with the fine topology. 
\end{pro}
\begin{proof}
By definition, $X^{an}$ is the set of multiplicative seminorms (or semivaluations in our terminology) on $A$ which extends $\nu$. Therefore, we have $X^{an}=\Hom_k(A,\mathbb{T})$ (as sets) from Lemma \ref{affinelem}. But, it follows from Proposition \ref{equivaelnceofcategories} that $\Hom_k(A,\mathbb{T})=\Hom_k(\Spec \mathbb{T}, X)=X(\mathbb{T})$, where $\nu:k\to \mathbb{T}$ is a fixed structural morphism. All it remains to show is that such a set bijection is a homeomorphism. But, this directly follows from Proposition \ref{affinetopology} and the definition of topology on $X^{an}$. 
\end{proof}

%\begin{rmk}
%One may wonder whether the classical result (the category of commutative rings is antiequivalent to the category of affine schemes) can be extended to the hyperring case, or not. As we mentioned earlier this is not true; however, in \cite[Proposition 4.31]{jun2015algebraic}, the author proved that the category of integral affine hyperring schemes is antiequivalent to the category of hyperdomains. 
%\end{rmk}

Let $X$ be a scheme of finite type over a complete non-Archimedean valued field $(k,\nu)$. Recall that the points of Berkovich analytification $X^{an}$ are in one-to-one correspondence with the set of equivalence classes of morphisms $\Spec L \to X$ for all valued extensions $L$ of $k$ such that two morphisms $\Spec L \to X$ and $\Spec L' \to X$ are equivalent if and only if there exists a valued extension $L''$ for both $L$ and $L'$ and a morphism $\Spec L'' \to X$ such that  the following diagram commutes:
\[
\begin{tikzcd}
\Spec L
\arrow{rd}[swap]{}
&\Spec L'' \arrow{d}{}
\arrow{l}
\arrow{r}
&\Spec L' \arrow{ld}{}
\\
&X
\end{tikzcd}
\]

The set of points of $X^{an}$ is also in one-to-one correspondence with the set of triples $(x,k(x),\mu)$, where $x$ is a point in $X$, $k(x)$ is the residue field at $x$, and $\mu$ is a valuation on $k(x)$ which extends $\nu$. With these interpretations, we have the following:

\begin{pro}\label{Berkovichsetbijection}
Let $X$ be a scheme of finite type over a complete non-Archimedean valued field $(k,\nu)$. Then there is a bijection (of sets) as follows:
\begin{equation}\label{setbijection}
X^{an}=X(\mathbb{T}).
\end{equation}
\end{pro}
\begin{proof}
It follows from Corollary \ref{analytification} that there is a one-to-one correspondence between the points of $X(\mathbb{T})=\Hom_k(\Spec\mathbb{T},X)$ and triples $(x,k(x),\tilde{\nu})$, where $x \in X$, $k(x)$ is the residue field at $x$, and $\tilde{\nu}$ is a homomorphism from $k(x)$ to $\mathbb{T}$ extending $\nu:k \to \mathbb{T}$. This is, in turn, in one-to-one correspondence with the points of $X^{\textrm{an}}$ as we explained above. 
\end{proof}

\begin{cor}\label{analytificationhomeo}
Let $X$ be a scheme of finite type over a field $k$ with a complete non-Archimedean valuation $\nu:k \to \mathbb{T}$. Then the analytification $X^{\textrm{an}}$ is homeomorphic to $X(\mathbb{T})$ which is equipped with the fine topology.
\end{cor}
\begin{proof}
We claim that the set-bijection in Proposition \ref{Berkovichsetbijection} is a homeomorphism. Let $i$ be the set-bijection in Proposition \ref{Berkovichsetbijection}. Since $\mathbb{T}$ is a topological hyperfield, we can apply Proposition \ref{openembedding} in this case and hence our proposition is reduced to the affine case. The result now follows from Proposition \ref{affinecase}.
\end{proof}

\begin{cor}\label{representable}
Let $k$ be a complete non-Archimedean valued field. 
Let $\mathcal{A}$ be the functor from the category of schemes of finite type over $k$ to the category of topological spaces sending $X$ to the underlying topological space $|X^{\textrm{an}}|$ of the Berkovich analytificaiton $X^{\textrm{an}}$. Then $\mathcal{A}$ is isomorphic to the functor $\Hom(\Spec \mathbb{T},-)$.
\end{cor}
\begin{proof}
For the notational convenience, we let $\mathcal{G}:=\Hom(\Spec \mathbb{T},-)$. For each scheme $X$ of finite type over $k$, we let $\eta_X:\mathcal{A}(X)=|X^{\textrm{an}}| \to \mathcal{G}(X)=X(\mathbb{T})$ be the homeomorphism in Corollary \ref{analytificationhomeo}. Then, for each $\mathfrak{p} \in |X^{\textrm{an}}|$, we have $\eta_X(\mathfrak{p}):\Spec \mathbb{T} \to X$ corresponding to a triple $(x,k(x),\tilde{\nu})$.

Now, let $f:X \to Y$ be a morphism of schemes of finite type over $k$. Then, $\mathcal{G}(f)(\eta_X(\mathfrak{p}))$ corresponds to the triple $(f(x),k(f(x)),\nu')$, where $\nu'$ is obtained by composing $k(f(x)) \to k(x)$ and $\tilde{\nu}$. In particular, one can easily see that the following diagram commutes:
\[
\begin{tikzcd}[row sep=large, column sep=1.5cm]
\mathcal{F}(X)\arrow{r}{\eta_X}\arrow{d}[swap]{\mathcal{F}(f)}
& \mathcal{G}(X) \arrow{d}{\mathcal{G}(f)} \\
\mathcal{F}(Y) \arrow{r}[swap]{\eta_Y} 
& \mathcal{G}(Y)
\end{tikzcd}
\]
This proves that the functors $\mathcal{F}$ and $\mathcal{G}$ are isomorphic. 	
%This is clear since, with the identification of \eqref{setbijection}, taking $\mathbb{T}$-rational points is functorial on $X$. In other words, if $X$ and $Y$ are schemes of finite type over $k$ and $f:X \to Y$ is a morphism of schemes, then this induces $f(\mathbb{T}):X(\mathbb{T})=X^{\textrm{an}} \to Y(\mathbb{T})=Y^{\textrm{an}}$ sending $p:\Spec \mathbb{T} \to X$ to $f\circ p:\Spec \mathbb{T} \to Y$. Furthermore, Proposition \ref{inducedlemma} shows that $f(\mathbb{T})$ is continuous. 
\end{proof}

\begin{rmk}
Let $(\Gamma,+)$ be a totally ordered abelian group, $\Gamma_{hyp}$ be the associated hyperfield, and $A$ be a commutative ring. A homomorphism $f:A \to \Gamma_{hyp}$ is the same thing as a prime ideal $\mathfrak{p} \in \Spec A$ and a homomorphism $\tilde{f}:A/\mathfrak{p} \to \Gamma_{hyp}$. The later data determines the Hahn analytification of $X=\Spec A$ as in \cite{foster2015hahn}, provided that $\Gamma_{hyp}$ is equipped with proper topology. This hints at that the analytification with a higher rank valuation case can be treated in the same way as Berkovich analytification, however, we do not pursue this case in this paper. 
\end{rmk}
\vspace{0.1cm}
\subsection{Hyperstructures of analytic groups}\label{hyperstructures of analytic groups}
In this section, we interpret several basic definitions and results in \cite[\S 5]{berkovich2012spectral} in terms of hyperstructures. To this end, we will mostly focus on the affine case. In what follows, we let $k$ be a complete non-Archimedean valued field.
\subsubsection{Hyperstructure of $G^{\textrm{an}}$}
 Let $G$ be a group scheme of finite type over $k$. The analytification $G^{\textrm{an}}$ of $G$ is a group object in the category of $k$-analytic spaces (see, \cite[\S 5]{berkovich2012spectral}). 
 
Let $p_i$ be the projections of $G^{\textrm{an}}\times_kG^{\textrm{an}}$ to the $i$ th factor for $i=1,2$ and $m:G^{\textrm{an}} \times_k G^{\textrm{an}} \to G^{\textrm{an}}$ be the multiplication of $G^{\textrm{an}}$. In general, the underlying space $|G^{\textrm{an}}|$ of $G^{\textrm{an}}$ is not a group itself, however, Berkovich introduces a `group-like' operation on $|G^{\textrm{an}}|$ as follows:

\begin{mydef}(\cite[\S 5]{berkovich2012spectral})\label{berkovichhypergroup}
Let $G$ be a group scheme of finite type over a field $k$ and $G^{\textrm{an}}$ be the analytification of $G$. One imposes a hyperoperation $\odot$ on $|G^{\textrm{an}}|$ as follows: for $g,h \in |G^{\textrm{an}}|$, 
\[
g\odot h:=\{f \in G^{\textrm{an}} \mid \exists w \in G^{\textrm{an}} \times_k G^{\textrm{an}}\textrm{ such that } p_1(w)=g,\textrm{ }p_2(w)=h,\textrm{ and } m(w)=f \}.
\]
\end{mydef}

Since $G^{\textrm{an}}$ is a group object, we have the inversion $i:G^{\textrm{an}} \to G^{\textrm{an}}$. For each $g \in G^{\textrm{an}}$, we let $g^{-1}:=i(g)$. Then we have the following. 

\begin{lem}\label{reverselem}
Let $(G^{\textrm{an}},\odot)$ be as above. If $x \in y\odot z$, then $x^{-1} \in z^{-1}\odot y^{-1}$. 
\end{lem}
\begin{proof}
If $x \in y \odot z$, then there exists $w \in G^{\textrm{an}}\times_k G^{\textrm{an}}$ such that $p_1(w)=y$, $p_2(w)=z$, and $m(w)=x$. Let $i:G^{\textrm{an}} \to G^{\textrm{an}}$ be the inversion and $\sigma: G^{\textrm{an}} \times_k G^{\textrm{an}} \to G^{\textrm{an}} \times_k G^{\textrm{an}}$ be the switch morphism. Let $w':=(i\times_k i)\circ \sigma(w)$. Then clearly one can see that $p_1(w')=z^{-1}$, $p_2(w')=y^{-1}$, and $m(w')=x^{-1}$ since $i \circ m=m\circ (i\times_k i)\circ \sigma$. 
\end{proof}

In \cite{berkovich2012spectral}, Berkovich actually proves the following proposition in our terminology. 

\begin{pro}
Let $G$ be a group scheme of finite type over $k$. Then $(|G^{\textrm{an}}|,\odot)$ is a hypergroup.
\end{pro}
\begin{proof}
It is proven in \cite[Proposition 5.1.1]{berkovich2012spectral} that $\odot$ is associative and there exists $e \in G^{\textrm{an}}$ such that $e\odot x=x\odot e=x$. Furthermore, it is also proven in \cite{berkovich2012spectral} that if $y \in g\odot x$ then $x \in g^{-1}\odot y$ (reversible condition). Therefore, we only have to show the following:
\begin{enumerate}
\item 
$e$ is the unique identity.
\item
For each $f \in G^{\textrm{an}}$, $f^{-1}$ is the unique inverse. 
\item
If $x \in y\odot z$ then $y \in x\odot z^{-1}$. 
\end{enumerate}
One can clearly see that $e$ is the unique identity element since if we have $e' \in G^{\textrm{an}}$ such that $e'\odot x=x$ for all $x \in G^{\textrm{an}}$, then we should have $e=e'\odot e=e'$. The uniqueness of inverses follows from the reversible condition: if $e \in g\odot h$, then $h \in g^{-1}\odot e=g^{-1}$. This implies that $h=g^{-1}$. Finally, if $x \in y \odot z$, then it follows from Lemma \ref{reverselem} that $x^{-1} \in z^{-1}\odot y^{-1}$. From \cite[Proposition 5.1.1]{berkovich2012spectral}, we have $y^{-1} \in z \odot x^{-1}$ and we derive $y \in x\odot z^{-1}$ by applying Lemma \ref{reverselem} again. 
\end{proof}

\begin{pro}
Let $G$ be a group scheme of finite type over $k$ and $H$ be a closed analytic subgroup of $G^{\textrm{an}}$. Then, for any $x,y \in H$, $x\odot y \subseteq H$. In particular, with the induced hyperoperation, $(H,\odot)$ is a sub-hypergroup of $(G^{\textrm{an}},\odot)$. 
\end{pro}
\begin{proof}
This is clear from the definition. 
\end{proof}

Let $G$ be a group scheme of finite type over $k$. Then we have the following canonical inclusions (of sets):
\begin{equation}\label{inclusion}
G(k) \xhookrightarrow{i} G \xhookrightarrow{j} G^{\textrm{an}}. 
\end{equation}

Berkovich's hyperoperation generalizes the classical group structure in the following sense.

\begin{pro}
Let $G$ be a group scheme of finite type over $k$ and let $G^{\textrm{an}}$ be the analytification of $G$. 
Let $i:G(k) \xhookrightarrow{} G$ and $j:G \xhookrightarrow{} G^{\textrm{an}}$ be the inclusions as in \eqref{inclusion}. Then for any $a,b \in G(k)$, we have 
\[
j(i(a*b)) \in j(i(a)) \odot j(i(b)),
\]
where $*$ is the group operation of $G(k)$. 
\end{pro}
\begin{proof}
This is clear as $a*b \in m(p^{-1}(a,b))$ for all $a,b\in G(k)$, where $p=(p_1,p_2)$. 
\end{proof}

\subsubsection{Berkovich's hyperstructure versus Connes and Consani's hyperstructure}\label{ccsection}
Let $G=\Spec A$ be an affine group scheme of finite type over $k$. In this case, we may use the identification $|G^{\textrm{an}}|=\Hom_k(A,\mathbb{T})$ to provide an algebraic definition of Berkovich's hyperoperation which is defined quite geometrically. We will also compare Berkovich's hyperoperation with the hyperoperation introduced by Connes and Consani in \cite{con4}. Let's first recall Connes and Consani's hyperstructure. 

\begin{mydef}\label{cchyper}
Let $A$ be a Hopf algebra over a field $k$ and $\Delta$ be the coproduct of $A$. By identifying, $X=\Spec A=\Hom_k(A,\mathbb{K})$, one imposes the following hyperoperation $\boxdot$ on $X$: for $f,g \in \Hom(A,\mathbb{K})$, 
\begin{equation}\label{hopgcc}
f \boxdot g:=\{h \in \Hom(A,\mathbb{K}) \mid h(a) \in \sum f(a_{(1)})g(a_{(2)}) \textrm{ for all } \Delta(a)=\sum a_{(1)}\otimes a_{(2)}\}.
\end{equation}
\end{mydef}

\begin{rmk}
The hyperoperation \eqref{hopgcc} makes sense for any hyperfield $H$ and $\Hom_k(A,H)$. We will consider later in this section the case when $H=\mathbb{T}$. 
\end{rmk}

In \cite{con4}, Connes and Consani compute the hyperoperation as in Definition \ref{cchyper} explicitly for an affine line and an algebraic torus. To be more precise, Connes and Consani prove the following:

\begin{mythm}(\cite{con4})\label{ccaffinetorus}
Let $X=\Spec k[T]$ be the affine line over $k=\mathbb{Q}$ and $\delta$ be the generic point of $X$. Let $G_H:=X -\{\delta\}$. Then $(G_H,\boxdot)$, where the hyperoperation $\boxdot$ as in Definition \ref{cchyper}, is a hypergroup. More precisely, we have the following isomorphism of hypergroups:
\[
G_H \simeq \overline{k}/\Aut_k(\overline{k}),  
\]
where $\overline{k}$ is considered as an additive group. One has a similar result for an algebraic torus $X=\Spec k[T,\frac{1}{T}]$ with $\overline{k}^\times$ as a multiplicative group and $k=\mathbb{F}_p$, the field with $p$ elements. 
\end{mythm} 

Inspired by Theorem \ref{ccaffinetorus}, in \cite{jun2016hyperstructures}, the author proves the following theorem. 

\begin{mythm}\label{mythmhyp}
Let $X=\Spec A$ be an affine group scheme of finite type over a field $k$. The hyperoperation $\boxdot$ on $X=\Hom_k(A,\mathbb{K})$ always satisfies the following properties:
\begin{enumerate}
\item 
$\exists !$ $e \in X$ such that $e\boxdot a=a\boxdot e$ for all $a \in X$. 
\item 
For each $a \in X$, there exists a canonical element $a^{-1} \in X$ (not necessarily unique) such that $e \in (a\boxdot a^{-1}) \bigcap (a^{-1} \boxdot a)$. 
\item
For $a,b,c \in X$, we have $((a\boxdot b) \boxdot c) \bigcap (a\boxdot (b \boxdot c)) \neq \emptyset$. 
\item
For $a,b,c \in X$, $a \in b \boxdot c$ if and only if $a^{-1} \in c^{-1} \boxdot b^{-1}$. 
\end{enumerate}
\end{mythm}

%\begin{rmk}
%The initial motivation of the current paper was to complete Theorem \ref{mythmhyp} (the trivial valuation case) by appealing to the results which Berkovich proved in \cite{berkovich2012spectral}. 
%\end{rmk}

Now, we consider the identification $|G^{\textrm{an}}|=\Hom_k(A,\mathbb{T})$, where $G=\Spec A$ is an affine group scheme of finite type over $k$. In other words, $A$ is a finitely generated Hopf algebra over $k$. In the remaining part of the subsection, we describe Berkovich's hyperoperation by means of the Hopf algebra structure of $A$. To this end, we first recall the following fact: Let $X=\Spec A$ and $Y=\Spec B$ be affine schemes of finite type over $k$. Then the product $X^{\textrm{an}} \times_k Y^{\textrm{an}}$ exists in the category of $k$-analytic spaces and in fact one has
\begin{equation}\label{fiberproduct}
X^{\textrm{an}} \times_k Y^{\textrm{an}}=(\Spec A\hat{\otimes}_k B)^{\textrm{an}}, 
\end{equation}
where $A\hat{\otimes}_k B$ is the completed tensor product of $A$ and $B$ over $k$ (see, \cite[\S 1]{berkovich2012spectral} or \cite[Appendix B]{bosch2014lectures} for the definition of complete tensor products). Since taking a fibered product commutes with the analytification, one may also write \eqref{fiberproduct} as follows:
\begin{equation}\label{fiber2}
X^{\textrm{an}} \times_k Y^{\textrm{an}}=(\Spec A\hat{\otimes}_k B)^{\textrm{an}}=(\Spec A \otimes_k B)^{\textrm{an}}. 
\end{equation}

\begin{rmk}
In the affine case, the fact that taking a fibered product commutes with the analytification directly follows from Proposition \ref{extension them}. 
\end{rmk}

Let $j_1:A \to A\otimes_k A$ be the homomorphism generated by sending $a$ to $a\otimes 1$ and $j_2:A \to A\otimes_k A$ be a homomorphism generated by sending $a$ to $1\otimes a$. The projection maps $p_i:|G^{\textrm{an}}\times_k G^{\textrm{an}}| \to |G^{\textrm{an}}|$ for $i=1,2$ can be rewritten as follows:
\begin{equation}\label{projection}
p_i:\Hom_k(A\otimes_kA, \mathbb{T}) \to \Hom_k(A,\mathbb{T}),\quad \nu \mapsto \nu\circ j_i, \quad \textrm{for }i=1,2. 
\end{equation}
Let $\Delta: A \to A\otimes_k A$ be the coproduct of $A$. Then the induced multiplication $m:|G^{\textrm{an}} \times_k G^{\textrm{an}}| \to G^{\textrm{an}}$ can be rephrased as follows:
\begin{equation}\label{delta}
m:\Hom_k(A\otimes_kA, \mathbb{T})  \to \Hom_k(A,\mathbb{T}), \quad \nu \to \nu\circ \Delta, 
\end{equation}

We define a hyperoperation $\star$ on $G^{\textrm{an}}=\Hom_k(A,\mathbb{T})$ as follows. 

\begin{mydef}\label{BerkovichhypHopf}
Let $A$ be a finitely generated Hopf algebra over $k$. For $g,h \in \Hom_k(A,\mathbb{T})$, we define the following set:
\[
g\star h:=\{f \in \Hom_k(A,\mathbb{T})\mid \exists \beta_f \in \Hom_k(A\otimes_k A,\mathbb{T})\textrm{ such that } \beta_f\circ j_1=g\textrm{ }, \beta_f\circ j_2=h, \textrm{ and } f=\beta_f\circ \Delta\}.
\]
\end{mydef} 

%\begin{rmk}
% One can easily see that the inclusion $i:G \hookrightarrow \Gan$ can be rewritten as $i:\Hom_k(A,\mathbb{K}) \hookrightarrow \Hom_k(A,\mathbb{T})$ which sends any homomorphism $A  \to \mathbb{K}$ to $A \to \mathbb{K} \hookrightarrow \mathbb{T}$. 
%\end{rmk}

\begin{pro} \label{HopfBerk}
Let $G=\Spec A$ be an affine group scheme of finite type over $k$ and $G^{\textrm{an}}$ be the analytification of $G$. Then, under identification of $G^{\textrm{an}}=\Hom_k(A,\mathbb{T})$, the hyperoperation defined by Berkovich agrees with the hyperoperation in Definition \ref{BerkovichhypHopf}.  
\end{pro}
\begin{proof}
This is clear. 
\end{proof}

Let $G=\Spec A$ be an affine group scheme of finite type over $k$. We denote by $*$ the hyperoperation on $G^{\textrm{an}}=\Hom_k(A,\mathbb{T})$ defined in Definition \ref{cchyper} by Connes and Consani. One may consider the hyperoperation of Berkovich as a refinement of Connes and Consani's hyperoperation in the following sense:

\begin{pro}\label{refinement}
We have the following inclusion: for any $g,h \in G^{\textrm{an}}$, 
\[
(g \star h) \subseteq (g*h)
\]
\end{pro}
\begin{proof}
Let $f \in g\star h$. Then there exists $\beta_f \in \Hom_k(A\otimes_kA,\mathbb{T})$ such that $\beta_f\circ j_1=g$, $\beta_f\circ j_2=h$, and $f=\beta_f \circ \Delta$. Let $a \in A$ and $\Delta(a)=\sum a_{(1)}\otimes a_{(2)}$. We have to show that $f(a) \in \sum g(a_{(1)})h(a_{(2)})$. But, since $f=\beta_f \circ \Delta$, we have
\begin{equation}\label{incl}
f(a) \in \sum \beta_f (a_{(1)}\otimes a_{(2)}).
\end{equation}
But, we have $g(a_{(1)})=\beta_f(j_1(a_{(1)}))=\beta_f(a_{(1)}\otimes 1)$ and $h(a_{(2)})=\beta_f(j_2(a_{(2)}))=\beta_f(1\otimes a_{(2)})$. Since $a_{(1)}\otimes a_{(2)}=(a_{(1)}\otimes 1)(1\otimes a_{(2)})$, \eqref{incl} becomes the following:
\begin{equation}
f(a) \in \sum \beta_f (a_{(1)}\otimes a_{(2)}) =\sum g(a_{(1)})h(a_{(2)}).
\end{equation}
This proves the desired result. 
\end{proof}

Recall that, since $\mathbb{K}$ is the final object in the category of hyperrings, we have a canonical projection $\pi:\mathbb{T} \to \mathbb{K}$ and the following reduction map:
\begin{equation}\label{reductionmap}
\pi:\Hom_k(A,\mathbb{T}) \longrightarrow \Hom_k(A,\mathbb{K}), \quad \varphi \mapsto \pi\circ \varphi. 
\end{equation}

\begin{pro}
Let $G=\Spec A$ be an affine group scheme of finite type over $k$. Let $*$ be the Connes and Consani's hyperoperation on $G^{\textrm{an}}$ and $\star$ be the Berkovich's hyperoperation on $G^{\textrm{an}}$, translated in terms of Hopf algebras as in Proposition \ref{HopfBerk}. Let $\pi:G^{\textrm{an}} \to G$ be the reduction map in \eqref{reductionmap}. Then we have the following inclusion: for $f,g \in \Hom_k(A,\mathbb{T})$, 
\[
\pi (f\star g) \subseteq \pi(f)* \pi(g).
\]
\end{pro}
\begin{proof}
Let $h \in f\star g$. We have to show that for any $a \in A$, $\Delta(a)=\sum a_{(1)} \otimes a_{(2)}$, 
\begin{equation}\label{ref2}
\pi\circ f (a) \in \sum (\pi\circ g(a_{(1)}))(\pi\circ h(a_{(2)}))=\sum \pi (g(a_{(1)})(h(a_{(2)}))
\end{equation}
However, it follows from Proposition \ref{refinement} that $f(a) \in \sum g(a_{(1)})h(a_{(2)})$ and hence
\[
\pi\circ f (a) \in \pi(\sum g(a_{(1)})h(a_{(2)})) \subseteq \sum \pi (g(a_{(1)})(h(a_{(2)})). 
\]
\end{proof}

\begin{rmk}
It is unclear whether or not two operations $\star$ and $*$ are different. It would be interesting to prove or disprove this. 
\end{rmk}

\subsubsection{Example: Affine line with a trivially valued ground field}

We explicitly compute the case for an affine line, provided that the ground field $k=\overline{\mathbb{Q}}$ is trivially valued. Note that the computation here is not new; Berkovich already computed this case in \cite[Example 5.1.4.]{berkovich2012spectral}, however, we will use the Hopf algebra structure to recompute this example. %We will make use of the following theorem:

%\begin{mytheorem}\cite{gruson1966theorie}
%Let $k$ be a complete non-Archimedean valued field and $M$, $N$ be complete normed $k$-vector spaces. Then a canonical map $M\otimes_k N \to M\hat{\otimes}_k N$ is injective. 
%\end{mytheorem}

In what follows, let $\mathbb{A}^1_k=\Spec k[T]$ and $i:\mathbb{A}^1_k(k) \hookrightarrow \mathbb{A}^{1,\textrm{an}}_k$. Then we can identify $\mathbb{A}^1_k(k)=k$. We can further identify each $q \in \mathbb{A}^1_k(k)$ as a homomorphism $f_q:k[T] \to k \to \mathbb{T}$ sending $a+bT$ to $a+bq$ and then to $\nu(a+bq)$, where $\nu$ is the trivial valuation on $k$. Let $(\mathbb{A}^{1,\textrm{an}}_k,\odot)$ be the hypergroup which we defined in the previous section (with $\odot$ defined by Berkovich). The following is well known (monomial valuations). 

\begin{lem}\label{monomialvaluation}
Let $k$ be a field and $a, b \in \mathbb{R}$. Then the following map, 
\begin{equation}\label{monomialexampleequation}
\beta:k[X,Y] \to \mathbb{T}, \quad 
\beta(\sum_{i,j} a_{ij}X^iY^j) =\left\{ \begin{array}{ll}
\max_{i,j}\{ia+jb\mid a_{ij}\neq 0\} & \textrm{if $\sum_{i,j} a_{ij}X^iY^j \neq 0$}\\
-\infty& \textrm{if $\sum_{i,j} a_{ij}X^iY^j =0$},
\end{array} \right.
\end{equation}
is an element of $\Hom_k(k[X,Y],\mathbb{T})$. 
\end{lem}

\begin{rmk}
One can observe that if $a<b <0$, then $\beta(X+Y)=b$. Also if $a=b$, then $\beta(X+Y)=a$. 
\end{rmk}

\begin{lem}\label{monomiallemma2}
Let $c\leq a \in \mathbb{R}$. Then there exists $\nu \in \Hom_k(k[X,Y],\mathbb{T})$ such that $\nu(X)=\nu(Y)=a$ and $\nu(X+Y)=c$. 
\end{lem}
\begin{proof}
Let $f:k[X',Y'] \to k[X,Y]$ be the isomorphism such that $f(X')=X$ and $f(Y')=Y+X$. It follows from Lemma \ref{monomialvaluation} that we have a semivaluation $\beta$ on $k[X',Y']$ such that $\beta(X')=a$ and $\beta(Y')=c$. Then we have $\beta(Y'+X')=\max\{a,c\}=a$. Let $g:=f^{-1}$. Then this defines a homomorphism $\nu:=\beta\circ g:k[X,Y] \to \mathbb{T}$. In particular, we obtain $\nu(X)=\beta(X')=a$, $\nu(Y)=\beta(Y'-X')=a$, and $\nu(X+Y)=\beta(Y')=c$ as desired. 
\end{proof}

\begin{lem}\label{classocallem}
There are exactly three cases of homomorphisms $\varphi:k[T] \to \mathbb{T}$ which extends the trivial valuation $\nu:k \to \mathbb{T}$ as follows:
\begin{enumerate}
\item 
$\varphi(f(T))=0$ $\forall f(T) \neq 0$.
\item
$\varphi(T) >0$ and $\varphi(f(T))=\deg(f) \cdot \varphi(T)$. 
\item
$\varphi (T) \leq 0$ and there exists an irreducible polynomial $f(T) \in k[T]$ such that $\varphi(f(T))<0$,
\end{enumerate} 
where $0=1_\mathbb{T}$. 
\end{lem}
\begin{proof}
This is an elementary result, for instance, see \cite[\S 3.1]{payne2015topology}.
\end{proof}

\begin{rmk}\label{equivalencermk}
The first case of Lemma \ref{classocallem} is the trivial valuation. In the second case, $\varphi$ is uniquely determined by the value $\varphi(T) \in (0,\infty)$; $\varphi(f)=\deg(f)\varphi(T)$. In particular, if $\varphi_1$ and $\varphi_2$ belong to the second case, then there exists $q \in (0,\infty)$ such that $\varphi_1=q\varphi_2$. 
\end{rmk}

Let $\delta$ be the trivial valuation. One can visualize the Berkovich analytification of an affine line in the following way depending on the above cases: 

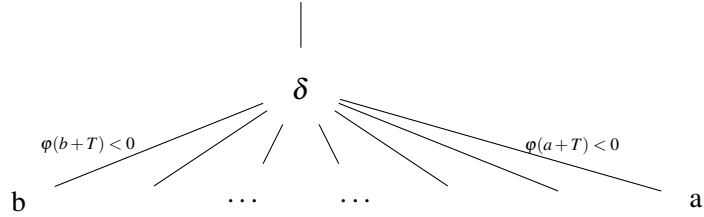
\begin{figure}[h]
\[
\begin{tikzpicture}[draw, circle,inner sep=7pt, -]
 \node {} 
    child { node {$\delta$}
    child [ missing ]
             child { node {b} edge from parent node[left] {$\Scale[0.6]{\varphi(b+T)<0}$}}  
             child { node {} }  
             child {node {\ldots}}
             child { node {\ldots} }  
             child { node {} }  
             child { node {} }  
             child { node {a} edge from parent node[right] {$\Scale[0.6]{\varphi(a+T)<0}$}}  
                                   };
    % Tree structure
 
\end{tikzpicture}
\]
\caption{\small{The Berkovich analytification $\mathbb{A}^{1,\textrm{an}}_k$ with respect to the trivial valuation on $k$.}}
\end{figure}

Let $\varphi:k[T] \to \mathbb{T}$ be a homomorphism as in the third case of Lemma \ref{classocallem}. Let $\mathfrak{p}_\varphi:=\{f \in k[T] \mid \varphi(f) <0\}$. Then one can easily show that $\mathfrak{p}_\varphi \in \Spec(k[T])$. Since $k$ is algebraically closed, we can find a unique linear polynomial (up to a product of constants) $g_b:=b+T$ which generates $\mathfrak{p}_\varphi$ for some $b \in k$. In particular, if $\varphi_1$, $\varphi_2 \in \Hom_k(k[T],\mathbb{T})$ belong to the third case such that $\mathfrak{p}_{\varphi_1}=\mathfrak{p}_{\varphi_2}$, then there exists $q \in (0,\infty)$ such that $\varphi_1=q\varphi_2$. Furthermore let $\varphi_b$ belong to the third case and $\mathfrak{p}_{\varphi_b}=<g_b>$, where $g_b:=b+T$. Then $\varphi_b$ is uniquely determined by $\varphi_b(g_b)$. In fact, we have the following. 

\begin{lem}
Let $\varphi:k[T] \to \mathbb{T}$ be a homomorphism such that $\varphi(T) \leq 0$ and $\mathfrak{p}_\varphi=<g_b>$, where $g_b:=b+T$. Then for any $f \in k[T]$, $\varphi(f)=r_b\varphi(g_b)$, where $r_b$ is the maximum natural number such that $g_b^{r_b}$ divides $f$.  
\end{lem}
\begin{proof}
We let $\oplus$ be the hyperaddition of $\mathbb{T}$. Let $f \in k[T]$ such that $g_b$ does not divide $f$. Since $k$ is algebraically closed, $f=l_1\cdots l_n$ for some linear polynomials $l_i$ which is not a constant multiple of $g_b$. Let $l_i=c+T$. As $l_i \not \in \mathfrak{p}_\varphi$, we have $\varphi(l_i) \geq 0$. However, $\varphi(l_i)=\varphi(c+T) \in \varphi(c) \oplus \varphi(T)$. If $c=0$, then we have $\varphi(l_i) =\varphi(T) \leq 0$ and hence $\varphi(l_i)=0$. If $c \neq 0$, then $\varphi(l_i) \in \varphi(c)\oplus \varphi(T)=0 \oplus \varphi(T)$. But, this still implies that $\varphi(l_i) \leq 0$ and hence $\varphi(l_i)=0$. This proves the lemma and furthermore this implies that $\varphi(g_b)$ uniquely determines $\varphi$. 
\end{proof}

We explicitly compute the case $g_0=T$, where $0=1_\mathbb{T}$.  

\begin{lem}\label{onebranch}
Let $x, y \in (-\infty, 0)$. Let $f_x,f_y$ be points of $\mathbb{A}^{1,\textrm{an}}_k$ such that $\mathfrak{p}_{f_x}=\mathfrak{p}_{f_y}=<g_0>$ and $f_x(g_0)=x$, $f_y(g_0)=y$. If $h \in f_x \odot f_y$, where $\odot$ is the Berkovich's hyperoperation, then $\mathfrak{p}_{h}=<g_0>$ and 
\begin{equation}\label{hypersum}
f_x\odot f_y =\left\{ \begin{array}{ll}
f_{\max\{x,y\}} & \textrm{if $x\neq y$}\\
\{f_t \mid t \in \left[-\infty,x\right]\}& \textrm{if $x=y$},
\end{array} \right.
\end{equation}
where $\left[-\infty,x\right]:=\{t\in \mathbb{T} \mid t\leq x\}$.
\end{lem}
\begin{proof}
We let $\oplus$ be the hyperaddition of $\mathbb{T}$. We first show that if $h \in f_x \odot f_y$ then $\mathfrak{p}_h=<g_0>$. In fact, there exists $\beta_h \in \Hom_k(A\otimes_k A, \mathbb{T})$ such that 
\begin{equation}\label{conditions}
\beta_h \circ j_1=f_x,\quad \beta_h \circ j_2 =f_y, \quad \textrm{and}\quad h=\beta_h \circ \Delta.
\end{equation}
It follows from \eqref{conditions} that
\[ 
\beta_h\circ j_1(g_0)=x, \quad \beta_h \circ j_2(g_0)=y,\quad \textrm{and}\quad h(g_0)=\beta_h \circ \Delta(g_0).\]
Hence we have:
\[
h(T)=\beta_h\circ \Delta(T)=\beta_h (T \otimes1+1\otimes T) \subseteq \beta_h (T \otimes1)\oplus \beta_h (1\otimes T) =x\oplus y. 
\]
Since $x,y <0$, this implies that $h(T)=h(g_0) <0$ and this forces $\mathfrak{p}_h$ to be generated by $g_0$. This proves that if $h \in f_x\odot f_y$ then $h(g_0) \in x\oplus y$ and $\mathfrak{p}_h=<g_0>$. 

Conversely, we have to show that if $z \in x\oplus y$ and $h \in \Hom_k(A,\mathbb{T})$ such that $h(g_0)=z$, then $h \in f_x \odot f_y$, i.e., we have to find $\beta_h :A \otimes_k A \to \mathbb{T}$ which satisfies the condition \eqref{conditions}. This is equivalent to find a homomorphism $\beta:k[X,Y] \to \mathbb{T}$ such that $\beta(X)=x$, $\beta(Y)=y$, and $\beta(X+Y)=z \leq \max\{x,y\}$. But, this has been proven in Lemma \ref{monomiallemma2}. 
\end{proof}

%Next, we consider the case when $f_x(T), f_y(T) \in \mathbb{R}_{>0}$. 

%\begin{lem}
%Let $x, y \in \mathbb{R}_{>0}$ and $f_x, f_y \in \Gan$ such that $f_x(T)=x$, $f_y(T)=y$. 
%\begin{enumerate}
%\item 
%If $x<y$, then $f_x\odot f_y=f_y$.
%\item 
%If $x=y$, then for any $h \in \Gan$ such that $h(T) >0$ belongs to $f_x\odot f_x$. 
%\item
%If $x=y$, then for any $h \in \Gan$ such that $h(T) < 0$ belongs to $f_x\odot f_x$. 
%\end{enumerate}
%\end{lem}
%\begin{proof}
%$(1)$, $(2)$, and $(3)$ directly follows from Lemma \ref{monomiallemma2}.
%\end{proof}

The following proposition shows that $(\mathbb{A}^{1,\textrm{an}}_k,\odot)$ contains a hypergroup which is isomorphic to a certain sub-hypergroup of $(\mathbb{T},\oplus)$. We define the following subset of $\mathbb{T}$:
\[
\mathbb{T}_{<0}:=\{a \in \mathbb{T} \mid a<0\}. 
\]
One can easily see that $\mathbb{T}_{<0}$ is a sub-hypergroup of $\mathbb{T}$, i.e., $\mathbb{T}_{<0}$ itself is a hypergroup with the induced hyperoperation. 

\begin{pro}
Consider the following set:
\[
H:=\{h \in \mathbb{A}^{1,\textrm{an}}_k \mid h(T) < 0 \}.
\]
Then $(H,\odot)$ is a hypergroup which is isomorphic to $\mathbb{T}_{<0}$. 
\end{pro}
\begin{proof}
Define an isomorphism, $\varphi:H \to \mathbb{T}_{<0}$ sending $f$ to $f(T)$. The result follows from Lemma \ref{onebranch}. 
\end{proof}

\begin{rmk}
Although we assume that $k$ is algebraically closed, we may use \cite[Lemma 5.1.2.]{berkovich2012spectral} to compute the non-algebraically closed field case. 
\end{rmk}

\begin{rmk}
Let $\delta$ be the trivial valuation and consider the following subset of $\mathbb{A}^{1,\textrm{an}}_k$:
\[
B:=\{h \in \mathbb{A}^{1,\textrm{an}}_k \mid h(T) \neq 0\}-\{\delta\}. 
\]
We may consider a new hyperoperation $\diamond$ on $B$ as follows:
\[
f \diamond g =\left\{ \begin{array}{lll}
f\odot g & \textrm{if $f(T)\neq g(T)$}\\
f\odot g& \textrm{if $f(T)=g(T) <0$}\\
(f\odot g) \bigcap B & \textrm{if $f(T)=g(T) >0$},
\end{array} \right.
\]
Then one can easily see that $(B,\diamond)$ is isomorphic to the hypergroup $(\mathbb{T}-\{0\},\oplus)$. To be precise, the isomorphism $\varphi$ is given by $\varphi(f)=f(T)$. It seems that the computation involving the trivial valuation $\delta$ is rather subtle as Connes and Consani already observed in \cite[\S7, \S8]{con4}. 
\end{rmk}
 
\section{Geometry of hyperfields in a view of real algebraic geometry} \label{spaceoforderings}

In this section, we prove that the functor $\Sper$ (a real algebraic analogue of the functor $\Spec$) is isomorphic to the functor $\Hom(-,\mathbb{S})$, where $\mathbb{S}$ is the hyperfield of signs. We first recall the definitions for \emph{real spectra and real schemes}. 

\begin{mydef}
Let $A$ be a commutative ring. By an \emph{ordering} on $A$, we mean a subset $P$ of $A$ such that
\begin{enumerate}
\item 
$P+P\subseteq P$ and $P\cdot P \subseteq P$.
\item
$a^2 \in P$ $\forall a \in A$ and $-1 \not \in P$. 
\item
$P\cup (-P) =A$, where $-P:=\{-a \mid a\in P\}$. 
\item
$P \cap (-P) \in \Spec A$ (the support of $P$). 
\end{enumerate}
\end{mydef}

\begin{lem}\label{orderinglemma}
Let $A$ be a commutative ring. There is a one-to-one correspondence between orderings on $A$ and elements of $\Hom(A,\mathbb{S})$. 
\end{lem}
\begin{proof}
Let $P$ be an ordering on $A$. Define a function $\varphi_P:A \to \mathbb{S}$ as follows:
\begin{equation}
\varphi_P(a) =\left\{ \begin{array}{lll}
1 & \textrm{if $a \in P\cap (-P)^c$},\\
0 & \textrm{if $a \in P \cap (-P)$}, \\
-1
& \textrm{if $a \in (-P)\cap P^c$}.
\end{array} \right.
\end{equation} 
One can easily see that $\varphi_P \in \Hom(A,\mathbb{S})$.

Conversely, for any $\varphi \in \Hom(A,\mathbb{S})$, let $P:=\varphi^{-1}(\{0,1\})$. Then $P$ is an ordering. Indeed, we obviously have $P+P \subseteq P$ and $P\cdot P \subseteq P$. For any $a \in A$, we have $\varphi(a^2)=(\varphi(a))^2 \in \{0,1\}$ and hence $a^2 \in P$. Furthermore, we have 
\[
\varphi(0)=\varphi(1+(-1))=0 \in \varphi(1)+\varphi(-1)=1+\varphi(-1).
\]
This implies that $\varphi(-1)=-1$ and hence $-1 \not\in P$. The condition that $P \cup (-P) =A$ is clear. Finally $P\cap (-P)=\ker(\varphi)$ and hence $P \cap (-P) \in \Spec A$. Clearly, these two constructions are inverses to each other. This shows the desired one-to-one correspondence.  
\end{proof}

Recall that a formally real field is a field $F$ equipped with an ordering. In other words, from Lemma \ref{orderinglemma}, a field $F$ is formally real if and only if the set $\Hom(F,\mathbb{S})$ is nonempty. 

\begin{lem}
Let $A$ be a commutative algebra over a formally real field $F$. We fix an ordering $P_F$ of $F$, i.e., we fix a homomorphism $\varphi_{P_F}:F \to \mathbb{S}$. Then there is a one-to-one correspondence between orderings on $A$ which contains $P_F$ and elements of $\Hom_F(A,\mathbb{S})$. 
\end{lem}
\begin{proof}
The proof is essentially same as the proof of Lemma \ref{orderinglemma}. 
\end{proof}

Now, we are ready to introduce a real spectrum. First, we define a real spectrum as a topological space and then review a \emph{Nash structure sheaf} which makes a real spectrum as a locally ringed space. For the details, we refer the readers to \cite{coste2001uniform} or \cite{fujita2003real}.  

\begin{mydef}\label{spectraltopology}
Let $A$ be a commutative ring. The \emph{real spectrum} $\Sper A$ of $A$ is the set of orderings on $A$ with topology, called the spectral topology, given by the basis of open subsets of the form $\{U(f)\}_{f \in A}$, where
\[
U(f):=\{P \in \Sper A \mid f \not \in -P\}. 
\]
\end{mydef}

\begin{rmk}
It follows from Definition \ref{spectraltopology} that, for $f \in A$, 
\[
U(-f):=\{P \in \Sper A \mid -f \not \in -P\}=\{P \in \Sper A \mid f \not \in P\}. 
\]
\end{rmk}

Next, we introduce a structure sheaf on $X=\Sper A$. For details, we refer the readers to \cite{coste2001uniform}.

Let $A$ be a commutative ring and $B$ be an \'{e}tale $A$-algebra. Then a map $f:A\to B$ induces a local homeomorphism $f^*:\Sper B \to \Sper A$ (see, \cite[Proposition 1.8]{scheiderer2006real}). 

One first defines a presheaf $\mathcal{A}$ on $X=\Sper A$ as follows: for an open subset $U$ of $X$, we let $\mathcal{A}(U)$ be the set of equivalence classes of triples $(B,s,b)$, where $B$ is an \'{e}tale $A$-algebra, $s:U \to \Sper B$ is a continuous section of the local homeomorphism $\Sper B \to \Sper A$, and $b \in B$. Two triples $(B,s,b)$ and $(C,t,c)$ are equivalent if and only if there exists a triple $(D,u,d)$ and $A$-algebra homomorphisms $f:B \to D$, $g:C \to D$ such that $f(b)=d=g(c)$ and the following diagram commutes:
\[
\begin{tikzcd}
\Sper B
&U \arrow{d}{u} \arrow{l}[swap]{s}
\arrow{r}{t}
&\Sper C 
\\
&\Sper D \arrow{lu}{f^*}
\arrow{ru}[swap]{g^*}
\end{tikzcd}
\]
The structure sheaf $\mathcal{O}_X$ of $X=\Sper A$ is the sheafification of $\mathcal{A}$. Then $(X,\mathcal{O}_X)$ is a locally ringed space and a real scheme is a locally ringed space which is locally isomorphic to a real spectrum. 
\begin{rmk}
In some special case, the structure sheaf $\mathcal{O}_X$ of a real spectrum $X=\Sper A$ can be defined as in the case of schemes by means of \emph{real strict localizations}. See, \cite{fujita2003real} for details.
\end{rmk}

Different from the case of schemes, in general, it is not true that the functors $\Sper$ and $\Gamma$ are inverses to each other. One only has the `\emph{idempotency property}'. 

\begin{pro}\cite[Theorem 23]{coste2001uniform}\label{idempotency}
Let $A$ be a commutative ring. Then 
\[
\Gamma (\Sper \Gamma (\Sper A)) \simeq \Gamma (\Sper A).
\]
In other words, $(\Gamma \circ \Sper)$ is an idempotent endofunctor on the category of commutative rings.  
\end{pro}

\begin{rmk}
We note that blue schemes, introduced by Lorscheid in \cite{lorscheid2012geometry}, satisfy a similar idempotency property as in Proposition \ref{idempotency}. Therefore, one may employ the idea of \emph{globalizations} in \cite{lorscheid2012geometry} in this setting. But, we do not pursue this perspective in the current paper. 
\end{rmk}

The following is well known.

\begin{lem}\label{reduction}
Let $A$ be a commutative ring, $X_r=\Sper A$, and $X=\Spec A$. Let $\red:X_r \to X$ be a function sending any ordering $P$ to $P\cap (-P)$. Then $\red$ is a well-defined continuous map. 
\end{lem}

Let $k$ be a complete non-Archimedean valued field, $A$ be a finitely generated $k$-algebra, and $X=\Spec A$. Recall that the Berkovich analytification $X^{\textrm{an}}$ of $X$ has the following decomposition:
\[
X^{\textrm{an}}=\bigsqcup_{\mathfrak{p} \in X} \nu_\mathbb{R}(\mathfrak{p}),
\] 
where $\nu_\mathbb{R}(\mathfrak{p})$ is the set of valuations on the residue field $k(\mathfrak{p})$ at $\mathfrak{p}$ extending that of $k$. Lemma \ref{reduction} provides a similar description for $\Sper A$ as follows:
\[
\Sper A=\bigsqcup_{\mathfrak{p} \in X} \nu_\mathbb{S}(\mathfrak{p}),
\] 
where $\nu_\mathbb{S}(\mathfrak{p})$ is the space of orderings of $k(\mathfrak{p})$ as in \cite{mars1} (see, Remark \ref{finalyrmk}).

Now, we impose topology on $\mathbb{S}$ by letting $\mathcal{T}:=\{\emptyset,\{-1\},\{1\},\{1,-1\},\mathbb{S}\}$ be the set of open subsets. 

\begin{lem}\label{sgnaffinecase}
Let $X=\Spec A$ be an affine scheme over $\mathbb{Z}$ and $X_r=\Sper A$. Then $X_r$ is homeomorphic to $X(\mathbb{S})$, equipped with the fine topology.
\end{lem}
\begin{proof}
Let $X=\Spec A$. Then, we have a set-bijection, $X_r=\Hom(A,\mathbb{S})=X(\mathbb{S})$. One can easily see that, with the topology $\mathcal{T}$, the fine topology on $X(\mathbb{S})$ is exactly the spectral topology of $X_r$ under the bijection of Lemma \ref{orderinglemma}. The result now simply follows from Proposition \ref{affinetopology}. 
\end{proof}

\begin{pro}
	The functor $\Sper$, from the category of rings to the category of topological spaces, is isomorphic to the functor $\Hom(-,\mathbb{S})$. 
\end{pro}
\begin{proof}
	One may apply a similar argument as in Corollary \ref{representable}. 
	%It follows from Lemma \ref{orderinglemma} that $\Sper(A)=\Hom(A,\mathbb{S})$ (as sets) and this is clearly functorial. 
\end{proof}

Let $X$ be a scheme over $\mathbb{Z}$. Then one can canonically associate a real scheme $X_r$ to $X$. Indeed, one may choose any affine open covering $\{U_i=\Spec A_i\}$ of $X$ and associate $(U_i)_r=\Sper A_i$ to each $i$, and then glue $\{(U_i)_r\}$ to obtain $X_r$ (see, \cite{scheiderer2006real}). As in the case of schemes and the Berkovich analytification, the underlying set of $X_r$ also has a nice description as a functor of points in the following sense: Recall that a real closed field is a field $k$ which is not algebraically closed, but the field extension $k(\sqrt{-1})$ is algebraically closed. Now, the underlying set of $X_r$ is the set of equivalence classes of rational points of $X$ over all real closed fields. Two rational points $f:\Spec k \to X$ and $g:\Spec k' \to X$ are equivalent if and only if there exists a real closed field extension $k''$ of $k$ and $k'$, together with a morphism $h:\Spec k'' \to X$, such that the following diagram commutes (see, \cite[\S 0.4]{scheiderer2006real}):
\[
\begin{tikzcd}
\Spec k
\arrow{rd}[swap]{f}
&\Spec k'' \arrow{d}{h}
\arrow{l}
\arrow{r}
&\Spec k' \arrow{ld}{g}
\\
&X
\end{tikzcd}
\]

\begin{lem}\label{reallemma}
Let $X$ be a scheme over $\mathbb{Z}$ and $X_r$ be the real scheme associated to $X$. Then there exists a one-to-one correspondence between the points of $X_r$ and the set of couples $(x,P_{k(x)})$, where $x \in X$, and $P_{k(x)}$ is an ordering on the residue field $k(x)$. 
\end{lem}
\begin{proof}
We may assume that $X$ is affine. Let $X=\Spec A$. Then we have $X_r=\Sper A=\Hom(A,\mathbb{S})$. Now, for each $f \in \Hom(A,\mathbb{S})$, we have a prime ideal $\mathfrak{p}=\ker(f)$ and this induces a homomorphism $\tilde{f}:\Frac(A/\ker(f)) \to \mathbb{S}$ and this is just an ordering on the residue field $k(\mathfrak{p})$. Conversely, for any prime ideal $\mathfrak{p}$ and an ordering $P_{k(\mathfrak{p})}$ on the residue field $k(\mathfrak{p})$ at $\mathfrak{p}$, we have a homomorphism $\pi:A\to k(\mathfrak{p})$ and $f:k(\mathfrak{p}) \to \mathbb{S}$. By composing these two, we obtain a homomorphism $f\circ \pi :A \to \mathbb{S}$. Clearly, these two constructions are inverses to each other. 
\end{proof}

\begin{pro}\label{realanaly}
Let $X$ be a scheme over $\mathbb{Z}$ and $X_r$ be the associated real scheme. Then $X_r$ and $X(\mathbb{S})$ are homeomorphic.
\end{pro}
\begin{proof}
Let $i:X_r \to X(\mathbb{S})$ be the set-bijection as in Lemma \ref{reallemma}. From the definition of $X_r$, we can choose an affine open covering $\{U_j\}$ of $X$, where $U_j=\Spec A_j$, such that $X_r$ can be covered by $\{V_j=\Sper A_j\}$. Consider the resection $i_{V_j}$ of $i$ to $V_j$. In this case, we have $i_{V_j}:V_j \to X(\mathbb{S})$ and, in fact, the image of $i_{V_i}$ is $U_i(\mathbb{S})$. From Proposition \ref{openembedding}, we may assume that $X$ is affine and the result follows from Lemma \ref{sgnaffinecase}. 
\end{proof}

\begin{cor}
Let $\mathcal{R}$ be the functor from the category of schemes to the category of topological spaces sending a scheme $X$ to the underlying topological space $|X_r|$ of the associated real scheme $X_r$. Then $\mathcal{R}$ is isomorphic to the functor $\Hom(\Spec \mathbb{S},-)$. 
\end{cor}
\begin{proof}
One may apply a similar argument as in Corollary \ref{representable}. 
\end{proof}

Conversely, given a real scheme $X_\mathfrak{R}$, one can associate a scheme $X_\mathfrak{R}^{red}$; fix an affine open covering $\{V_i=\Sper A_i\}$ of $X_\mathfrak{R}$, we may associate $U_i=\Spec A_i$ for each $i$ and glue these to obtain a scheme $X_\mathfrak{R}^{red}$ over $\mathbb{Z}$.

\begin{pro}\label{reduct}
Let $X_\mathfrak{R}$ be a real scheme. Then the reduction map $\mathbf{red}:X_\mathfrak{R}\to X_\mathfrak{R}^{red}$ is continuous. Moreover, the real scheme $(X_\mathfrak{R}^{red})_r$ associated the the scheme $X_\mathfrak{R}^{red}$ (as in Proposition \ref{realanaly}) is homeomorphic to $X_\mathfrak{R}$. 
\end{pro}
\begin{proof}
We may assume that $X_\mathfrak{R}$ is affine and the first statement follows from Lemma \ref{reduction}. The second statement is clear from the definition of $(X_\mathfrak{R}^{red})_r$. 
\end{proof}

\begin{cor}
Let $X_\mathfrak{R}$ be a real scheme. Then $(X_\mathfrak{R}^{red})_r(\mathbb{S})$ is homeomorphic to $X_\mathfrak{R}$. 
\end{cor}
\begin{proof}
This directly follows from Propositions \ref{realanaly} and \ref{reduct}. 
\end{proof}

\begin{rmk}
Although we only state the case when a scheme is over $\mathbb{Z}$, one can easily prove similar results for a scheme over a formally real field $F$. Note that $F$ should be a formally real field since otherwise, there is no homomorphism from $F$ to $\mathbb{S}$. 
\end{rmk}

\begin{rmk}\label{finalyrmk}
In \cite{hochster1969prime}, Hochster characterized topologically the essential image of the functor $\Spec$ by introducing the notion of spectral spaces. One has a similar result in real algebraic geometry. To be precise, let $F$ be a formally real field. Then the real spectrum $X_F=\Sper F$ is called the space of orderings on $F$. Then $X_F$ becomes a Stone space (or Boolean space), i.e., compact, Hausdorff, and totally disconnected space. Conversely, it is proved by T.~Craven in \cite{craven1975boolean} that for any Stone space $X$, there exists a formally real field $F$ such that $X$ is homeomorphic to $X_F$. Finally, we remark that if the characteristic of $F$ is not equal to $2$, then $X_F$ can be realized as the set of minimal prime ideals of $W(F)$ (the Witt ring of quadratic forms of $F$) and furthermore $X_F$ is homeomorphic to the set of minimal prime ideals of $W(F)$ equipped with the Zariski topology. For more details, we refer the readers to \cite{lam2005introduction}. It would be interesting to investigate these perspectives in terms of geometry of $\mathbb{S}$ following the ideas in \cite{mars1},  \cite{mar2}, and \cite{gladki2017witt}. 
\end{rmk}

%\begin{rmk}
%Following the idea of Hochster, it would be also interesting to characterize topological spaces which arises as sets of rational points over schemes over hyperfields. For instance, all spectral spaces arise in this way. 
%\end{rmk}

\begin{rmk}
In \cite{con3}, Connes and Consani introduce the notion of tensor products for $\mathbb{K}$ and $\mathbb{S}$ in a certain restricted case. For instance, if $X=\Spec A$ is an affine scheme over a field $k$, then `a scalar extension' $X_\mathbb{K}$ is defined to be $\Spec (A/k^\times)$. When $k$ is a formally real field with a fixed homomorphism $\varphi:k \to \mathbb{S}$, $X_\mathbb{S}=\Spec (A/P)$, where $P$ is the ordering corresponding to $\varphi$. One may develop this approach further to incorporate the notion of tensor products with hyperfields with the approach taken in the current paper. 
\end{rmk}

\bibliography{geobib}\bibliographystyle{alpha}

\end{document}